\documentclass[a4paper,11pt]{article}
\usepackage[utf8]{inputenc}
\usepackage[T1]{fontenc}

\usepackage{amsthm}
\usepackage{tikz}
\usepackage{mathrsfs}  
\usepackage{fullpage}

\usepackage{thmtools}
\usepackage{mathtools}
\usepackage{amssymb}
\usepackage[nomath]{lmodern}
\usepackage{graphicx}
\usepackage{pgf,tikz,tkz-graph,subcaption}
\usetikzlibrary{arrows,shapes}
\usetikzlibrary{decorations.pathreplacing}
\usepackage{tkz-berge}
\usepackage{enumitem}
\usepackage[normalem]{ulem}
\usepackage{hyperref,comment}
\hypersetup{colorlinks = true, linkcolor = blue, citecolor = blue, urlcolor = blue}

\allowdisplaybreaks

\usepackage[margin=1in]{geometry}
\usepackage[textsize=small]{todonotes}

\usepackage{bookmark}

\newtheorem{innercustomthm}{Theorem}
\newenvironment{customthm}[1]
  {\renewcommand\theinnercustomthm{#1}\innercustomthm}
  {\endinnercustomthm}

\newtheorem{theorem}{Theorem}
\newtheorem*{theorem*}{Theorem}

\newtheorem{corollary}[theorem]{Corollary}
\newtheorem{lemma}[theorem]{Lemma}

\newtheorem{proposition}[theorem]{Proposition}

\newtheorem{defi}[theorem]{Definition}
\newtheorem{claim}[theorem]{Claim}
\newcommand*{\myproofname}{Proof}
\newenvironment{claimproof}[1][\myproofname]{\begin{proof}[#1]}{\end{proof}}
\newtheorem*{claim*}{Claim}
\newtheorem{examp}[theorem]{Example}

\newtheorem{question}[theorem]{Question}
\DeclareMathOperator{\DB}{DB}
\DeclareMathOperator{\SSS}{SS}
\DeclareMathOperator{\Sp}{Sp}
\DeclareMathOperator{\ST}{ST}
\newcommand{\sapm}{SAPM}

\newcommand{\apm}{APM}
\newcommand{\ksapm}{$k$-SAPM}
\newcommand{\kapm}{$k$-APM}

\newcommand*{\ceilfrac}[2]{\mathopen{}\left\lceil\frac{#1}{#2}\right\rceil\mathclose{}}
\newcommand*{\floorfrac}[2]{\mathopen{}\left\lfloor\frac{#1}{#2}\right\rfloor\mathclose{}}

\newcommand*{\abs}[1]{\lvert #1\rvert}

\title{Trees maximizing the number of almost-perfect matchings}

\author{Stijn Cambie \and Bradley McCoy \and Gunjan Sharma \and Stephan Wagner \and Corrine Yap}
\date{\today}

\begin{document}

\maketitle

\begin{abstract}

    We characterize the extremal trees that maximize the number of almost-perfect matchings, which are matchings covering all but one or two vertices, and those that maximize the number of strong almost-perfect matchings, which are matchings missing only one or two leaves. We also determine the trees that minimize the number of maximal matchings. We apply these results to extremal problems on the weighted Hosoya index for several choices of vertex-degree-based weight function.
\end{abstract}

\begin{section}
{Introduction}\label{sec:intro}
\end{section}

A matching in a graph $G$ is a set of edges such that no two edges share a vertex. Matching theory has a rich history with many applications. In particular, extremal problems and enumeration problems involving matchings have been studied extensively. A maximum matching can be the best solution to a variety of problems, such as scheduling and task assignment problems or matching donors with receivers~\cite{Gale-Shapley-1962,LP86}.

Computing the number of matchings in a graph, also called the Hosoya index, is \#P-complete in general. In this paper, we focus on enumeration of matchings in trees.
It is known that the number of matchings in the $n$-vertex path graph $P_n$ is equal to $F_n$, the $n$th Fibonacci number (where $F_1=1,F_2=2$), and that this is the maximum over all $n$-vertex trees~\cite{Gutman80}. In fact, the star and the path respectively minimize and maximize not only the number of matchings but also the number of induced matchings~\cite{KKKL17}, and the number of matchings of a fixed size~\cite[Thm.~4.7.3]{WH19}. In the last case, we show that the path is, in fact, the unique maximizer. For a more comprehensive background on matching enumeration problems, see e.g. the survey of Wagner and Gutman~\cite{WG10}, or the book by Wagner and Wang~\cite{WH19}.

To prove that the path uniquely maximizes the number of matchings of a fixed size, we require results on almost-perfect matchings which are of independent interest. 
Perfect matchings are central to the study of matching theory---Hall's matching theorem and Tutte's theorem are perhaps the most notable examples \cite{Hall1935, Tutte47}. However, the notion of a matching that is ``close to perfect'' has been less prevalent. When the number of vertices in a graph is odd, an almost-perfect matching is a matching that avoids a single vertex. This is also called a near-perfect matching, and enumeration problems for other classes of graphs (such as factor-critical graphs and grid graphs) have been considered~\cite{doroslovavcki2019, Liu02, perepechko2019counting}.
When the number of vertices is even, an almost-perfect matching avoids two vertices. We characterize the extremal trees in both cases and prove the following:

\begin{theorem}\label{thr:odd_apm}
    If $n$ is odd, a tree $T$ of order $n$ has at most $\frac{n+1}{2}$ almost-perfect matchings.
    Equality holds if and only if $T$ is a $1$-subdivision of a tree of order $\frac{n+1}{2}.$
\end{theorem}

\begin{theorem}\label{thr:even_apm}
	If $n$ is even, a tree $T$ of order $n$ has at most $\binom{ \frac{n}{2}+1}{2}=\frac{n(n+2)}{8}$ almost-perfect matchings.
	\begin{itemize}
	\item For $n > 4$, equality holds if and only if $T$ is the path $P_n$.
	\item For $n=4$, equality holds for both $P_4$ and $S_4$.
    \end{itemize}
\end{theorem}

We consider the same questions for strong almost-perfect matchings, which are matchings that cover all vertices except one or two leaves, and again find precise characterizations of the extremal graphs. When the order is odd, the number of strong almost-perfect matchings is maximized by the basic spider (a star with all but at most one edge subdivided once, as presented in Figure~\ref{fig:spiders}); this is the same construction that, for example, maximizes the number of maximal independent sets~\cite{Wilf86}.

\begin{theorem}\label{thr:sapmcombined}
Let $T$ be a tree of order $n\ge 28$.
Then the number of strong almost-perfect matchings in $T$ is at most $\frac{n-1}{2}$ if $n$ is odd and $\floorfrac{(n-4)^2}{12}$ if $n$ is even.
Equality holds if and only if $T$ is a spider (when $n$ is odd) or a balanced spider-trio (when $n$ is even).
\end{theorem}

In fact, we are able to give a complete picture of the bounds and extremal trees for all values of $n$, but postpone the precise statements until Section~\ref{sec:sapm}. 

The trees that maximize the number of maximal matchings have been previously determined by G\'{o}rska and Skupie\'{n}~\cite{GS07}.
Surprisingly, the characterization of the trees that minimize the number of maximal matchings has not appeared in the literature.
We show that the extremal trees are precisely the most basic spiders again, or a special spider if $n$ is odd. 

\begin{theorem} \label{thr:minmaxmatchings}
   A tree $T$ of order $n$ has at least $\lceil \frac n2 \rceil $ maximal matchings. 
    Equality holds if and only if $T$ is a spider $\Sp_n$, or an odd special spider $\SSS_n$.
\end{theorem}

In Section~\ref{sec:HosoyaIndex}, we consider applications to the weighted Hosoya index $Z_\phi(T)$, where each edge is assigned a weight by the function $\phi$ and the weight of a matching is the product of its edge weights. Motivated by recent results by Cruz, Gutman, and Rada~\cite{CRUZ2022}, we consider certain choices of vertex-degree-based weight functions and prove that while the minimizer is the star in all cases that we consider, the maximizer can be the path, the double-broom, or a type of spider.

The outline of the paper is as follows: we start with mentioning some basics in Section~\ref{subsec:def_not}. In Section~\ref{sec:ExtrTrees_mk}, we prove Theorems~\ref{thr:odd_apm} and~\ref{thr:even_apm}. In Section~\ref{sec:sapm}, we prove Theorem~\ref{thr:sapmcombined}. In Section~\ref{sec:minmaxmatching}, we prove Theorem~\ref{thr:minmaxmatchings}. In Section~\ref{sec:HosoyaIndex}, we prove results related to the weighted Hosoya index. We conclude by discussing some further questions to pursue in Section~\ref{sec:future}.

\section{Preliminaries}\label{subsec:def_not}

In this section, we introduce definitions and notation that will be used throughout the paper.
Given a graph $G=(V,E)$, a \emph{matching} $M$ in $G$
is a set of edges such that no two edges share a common
vertex. A matching is \emph{maximal} if it is not a subset of any other matching.
A matching is a \emph{maximum matching} if it has the largest possible cardinality.
A \emph{perfect} matching is a matching that includes all vertices. 

The subgraph $G[V']$ of $G=(V,E)$ induced by a set $V'\subset V$ is the graph with vertex set $V'$ and edge set $E \cap \binom{V'}{2}.$
Two vertices $u,v$ are adjacent (neighbors of each other) if $uv=(u,v) \in E$. This adjacency will also be denoted by $u \sim v.$

The {\em diameter} of a graph $G$, 
denoted $diam(G)$, is the maximum distance between any pair of vertices over all pairs in $G$. When we say ``a diameter of $G$,'' we mean a path with length equal to $diam(G)$. A {\em $1$-subdivision of an edge} $uv$ is obtained by adding a new vertex $x_{uv}$ and replacing $uv$ with the path $\{u, x_{uv}, v\}$. A {\em $1$-subdivision of a graph} is a graph all of whose edges are ($1$-)subdivided.

\begin{defi}
    In a tree of order $n$, a matching $M$ is an {\em almost-perfect matching} (\apm) if it covers all but one vertex when $n$ is odd and all but two vertices when $n$ is even. The vertices not covered by $M$ are called the {\em avoided vertices}. $M$ is a {\em strong almost-perfect matching} (\sapm) if the avoided vertices are leaves. 
\end{defi}

Let $\mathcal M(G)$ be the set of all matchings of $G$, $\mathcal M_k(G)$ be the set of matchings in $G$ of size $k$, and  $m_k(G) = |\mathcal M_k(G)|$ with $m_0(G) = 1$. 
The \emph{Hosoya index} of $G$ is the number of matchings of $G$, which can be written as 
$$Z(G)= |\mathcal M(G)| = \Sigma_{k\geq0} m_k(G).$$ 
The name of the index refers to Hosoya~\cite{HOSOYA71} who observed that if $G$ is a molecular graph, the Hosoya index is correlated with its chemical properties.

Let $(G,\omega)$ denote a graph with weight function
$\omega:E\to\mathbb{R}^+$. The weight of a matching $M$ is then the product of the weights of the edges in $M$, 
and the \emph{weighted Hosoya index} is
$$Z(G, \omega) = \sum_{M \in \mathcal{M}(G)}\left(\prod_{e \in M} \omega(e)\right).$$

Note that $\emptyset \in \mathcal{M}(G)$ for all $G$; by convention we take the weight of the empty matching to be $1$, so when viewed as a polynomial in $\{\omega(e)\}_{e \in E(G)}$, the weighted Hosoya index always has constant term $1$. 
As an aside, observe that this coincides with the definition of the (multivariate) monomer-dimer partition function from statistical physics when $\omega(e)$ is a positive constant $\lambda_e$ for each edge $e$. 

Let $e=uv$ with $\deg(u)=i$
and $\deg(v)=j$. The weight function
is \emph{vertex-degree-based} if $\omega(e)=\phi(i,j)$
where $\phi$ is a function such that $\phi(i,j) = \phi(j,i)$.
In this case, we denote $Z(G,\omega)$ by $Z_{\phi}(G).$

We finish this section with some statements that we will use in our proofs throughout. We will refer to the fact $\sum_{v \in V(G)} \deg(v) = 2|E(G)|$ as the ``handshaking lemma.'' The following is also a well-known statement:

\begin{lemma}\label{lem:treepm}
Every tree has at most one perfect matching.
\end{lemma}

\begin{proof}
    Suppose $M$ and $M'$ are two perfect matchings of $T$. In the graph $(V(T), M \cup M')$, every component must be either a single edge (contained in both $M$ and $M'$) or a cycle. As $T$ is a tree, it contains no cycles and so we must have $M = M'$.
\end{proof}

For two non-increasing sequences $x_1 \geq x_2 \geq \cdots \geq x_k$ and $y_1 \geq y_2 \geq \cdots \geq y_k$, we say $(x_i)_{i=1}^k$ {\em majorizes} $(y_i)_{i=1}^k$ if and only if  $\sum_{i=1}^j x_i \geq \sum_{i=1}^j y_i$ for every $1 \le j \le k$ and equality holds for $j=k$.
Analogously, for two non-decreasing sequences $x_1 \leq x_2 \leq \cdots \leq x_k$ and $y_1 \leq y_2 \leq \cdots \leq y_k$, we say $(x_i)_{i=1}^k$ {\em majorizes} $(y_i)_{i=1}^k$ if and only if  $\sum_{i=1}^j x_i \leq \sum_{i=1}^j y_i$ for every $1 \le j \le k$ and equality holds for $j=k$.

Observe that if $(x_i)_{i=1}^k$ and  $(y_i)_{i=1}^k$ are non-increasing, then the sequences $(x_i')_{i=1}^k$ and $(y_i')_{i=1}^k$ defined by $x_i' = x_{k-i+1}$ and $y_i' = y_{k-i+1}$ are non-decreasing, and $(x_i)_{i=1}^k$ majorizes $(y_i)_{i=1}^k$ if and only if $(x_i')_{i=1}^k$ majorizes $(y_i')_{i=1}^k$.

% Observe that an equivalent definition is the following: $(x_i)_{i=1}^k$ where $x_1 \leq x_2 \leq \cdots \leq x_k$ majorizes $(y_i)_{i=1}^k$, $y_1 \leq y_2 \leq \cdots \leq y_k$, if and only if $\sum_{i=1}^j x_i \leq \sum_{i=1}^j y_i$ for all $1 \leq j \leq k$ with equality when $j = k$.

Our main use of this notion is the following inequality.

\begin{lemma}[Karamata's inequality, \cite{Kar32}]\label{lem:karamata}
    If $(x_i)_{i=1}^k$ is a sequence of real numbers that majorizes $(y_i)_{i=1}^k$, and $f$ is a real-valued convex function, then $\sum_{i=1}^k f(x_i) \geq \sum_{i=1}^k f(y_i)$. This inequality is strict if the sequences are not equal and $f$ is a strictly convex function. If $f$ is concave, the reverse inequality holds.
\end{lemma}

\begin{section}
{Maximum number of matchings of a fixed size}\label{sec:ExtrTrees_mk}
\end{section}

In this section, we characterize the trees $T$ of order $n$ maximizing the number of almost-perfect matchings and use this to characterize those maximizing the number of matchings of a fixed size $k$.
When $n=2k$, any tree with a perfect matching is extremal since the perfect matching must be unique. While it is straightforward to check if a given tree has a perfect matching, there is not a clear characterization of the set of all trees with perfect matchings (there are exponentially many of them~\cite{Simion91}).

When $n=2k+1$, we prove that there are many extremal trees maximizing the number of almost-perfect matchings. For convenience, we recall the theorem statement:
\begin{customthm} {\bf \ref{thr:odd_apm}} If $n$ is odd, a tree $T$ of order $n$ has at most $\frac{n+1}{2}$ almost-perfect matchings.
    Equality holds if and only if $T$ is a $1$-subdivision of a tree of order $\frac{n+1}{2}.$
\end{customthm}

\begin{proof}
    Recall that every tree has a unique bipartition, so let $T=(X \cup Y,E)$. If $T$ has an almost-perfect matching, then the bipartition classes must differ by $1$ and the larger one, without loss of generality $X$, contains precisely $\frac{n+1}{2}$ vertices. This immediately implies the upper bound, since $T\setminus \{y\}$ for $y \in Y$ cannot have a perfect matching.

    Next, we prove the characterization of the extremal trees.
    Let $T'$ be an arbitrary tree of order $\frac{n+1}{2}$.
    Let $T$ be a 1-subdivision of $T'$, with $Y$ the set of added vertices.
    For every $v \in V(T')$, we can construct a perfect matching of $T \setminus \{v\}$: viewing $T$ as a tree rooted at $v$, match each $u \in V(T')\setminus \{v\}$ with its neighbor (in $Y$) on the unique path from $u$ to $v$.
    Hence any subdivision of a tree of order $\frac{n+1}{2}$ contains $\frac{n+1}{2}$ almost-perfect matchings.

    We prove the other direction by induction on $n$.
    The cases $n \in \{1,3\}$ are straightforward, since $P_3$ is the only tree of order $3$ and is the $1$-subdivision of $P_2$.
    Thus, assume the statement is true for $n-2,$ where $n \ge 5$.
    Let $T=(X \cup Y,E)$ be a tree that attains equality and without loss of generality, let $|X| = \frac{n+1}{2}$. Then for every choice of $v \in X$, $T \setminus \{v\}$ has a perfect matching.
    Observe that every leaf of $T$ must be in $X$. If there is a leaf in $Y$, then its unique neighbor in $X$ cannot be avoided by an almost-perfect matching, since any such matching would also avoid the leaf itself.
    This implies that no two leaves can have the same neighbor. Indeed, if $x_1, x_2 \in X$ are leaves, then for any $v \in X$ such that $v \neq x_1, x_2$, the almost-perfect matching that avoids $v$ would have to cover both leaves $x_1$ and $x_2$, a contradiction.
    
    Now let $x \in X$ be a leaf such that its unique neighbor $y$ has degree 2 (such a leaf must exist, e.g.~one of the ends of a diameter). Then $T':=T \setminus \{x, y\}$ is a tree.  Furthermore, for any $v \in X \setminus \{x\}$, we know that $T \setminus \{v\}$ has a perfect matching, which must contain the edge $xy$. Thus, $T'$ contains $\frac{n-1}{2}$ almost-perfect matchings so we may apply the inductive hypothesis to conclude that $T'$ is a $1$-subdivision of a tree, and hence $T$ is as well.
\end{proof}

Next, we will show that when $n=2k+2$, the path is extremal in all non-trivial cases. We start with a lemma that will be applied in the proof of this and some further theorems.

\begin{lemma}\label{lem:leavesatdistance2=>n_bound}
    If $n$ is even and $T$ is a tree of order $n$ with at least one pair of leaves sharing a common neighbor, then $T$ has at most $n-1$ almost-perfect matchings.
\end{lemma}
\begin{proof} Recall that for $n$ even, an almost-perfect matching $M$ covers all but two vertices, and we call these vertices the avoided pair of $M$. 
Let $T$ be a tree of order $n$, and let $u$ and $u'$ be the two leaves with a common neighbor (siblings).
No matching can cover both leaves, so for every almost-perfect matching $M$ of $T$, at least one of $u$ and $u'$ is contained in the avoided pair of $M$.

By Theorem~\ref{thr:odd_apm}, $T\setminus \{u\}$ and $T\setminus  \{u'\}$ have both at most $\frac{n}{2}$ almost-perfect matchings.
Equality is only possible if $T\setminus \{u\}$ and $T\setminus  \{u'\}$ are $1$-subdivisions of a tree. Since the perfect matching of $T\setminus \{u, u'\}$ has been counted twice, $T$ has at most $n-1$ almost-perfect matchings.
Because $T\setminus \{u\}$ and $T\setminus \{u'\}$ are isomorphic, we also know that equality is attained if and only if $T$ is a $1$-subdivision of a tree of order $\frac{n}{2}$ where one copy of a leaf is added (i.e. a copy $u'$ of a leaf $u$ with the same neighbor).
\end{proof}

We use this to prove the case for $n$ even:
\begin{customthm} {\bf \ref{thr:even_apm}}
	If $n$ is even, a tree $T$ of order $n$ has at most $\binom{ \frac{n}{2}+1}{2}=\frac{n(n+2)}{8}$ almost-perfect matchings.
	\begin{itemize}
	\item For $n > 4$, equality holds if and only if $T$ is the path $P_n$.
	\item For $n=4$, equality holds for both $P_4$ and $S_4$.
    \end{itemize}
\end{customthm}

\begin{proof}
    One can check by hand that for $n = 4$ both $P_4$ and $S_4$ have three almost-perfect matchings, for $n=6$ the path $P_6$ has six almost-perfect matchings, and in both cases these are the only examples attaining the maximum.  

	 Now let $n \ge 8$ and suppose the claim holds for all even numbers less than or equal to $n-2$. 
	 
	 Since $n < \frac{n(n+2)}{8}$, by Lemma~\ref{lem:leavesatdistance2=>n_bound} we may assume there are no two leaves with a common neighbor. Taking a diameter of the tree, we note that there is an edge $uv$ with $u$ a leaf and $v$ its neighbor of degree $2$. Then either $uv$ is contained in the almost-perfect matching, or the avoided pair contains $u$. We count the number of matchings in each case and take the sum to get the total number of almost-perfect matchings in $T$.
	 
	 For the first term, the number of almost-perfect matchings containing $uv$ is equal to the number of almost-perfect matchings in $T \setminus \{u,v\}$ which by the inductive hypothesis is at most $\frac{n(n-2)}{8}$.
	 For the second term, the number of avoided pairs containing $u$ is at most the number of almost-perfect matchings in $T \setminus \{u\}$, which is bounded by $\frac{(n-1)+1}{2}$ by Theorem~\ref{thr:odd_apm}.
	 Since $\frac{n(n-2)}{8}+\frac{(n-1)+1}{2}=\frac{n(n+2)}{8}$, the first statement of the theorem holds by induction.
	 
	 To characterize equality, we see by the inductive hypothesis that equality holds for the first term if and only if $T \setminus \{u, v\} = P_{n-2}$. By Theorem~\ref{thr:odd_apm}, equality holds for the second term if and only if $T \setminus \{u\}$ is a 1-subdivision of a tree of order $\frac n2$. Since $v$ must have degree 2, these imply that equality holds overall if and only if $T = P_n$.
\end{proof}

Another way of stating Theorem~\ref{thr:even_apm} is that $P_n$ maximizes $m_k(T)$ when $n = 2k+2$ and $k \geq 2$. We can extend this to larger values of $n$. Note that $m_1(T)=n-1$ for every tree $T$ and so it has been omitted in the statement. 

\begin{theorem}\label{thr:maxmatch}
    If $k \geq 2$ and $T$ is a tree of order $n$ with $n \ge 2k+2$, then 
    $m_k(T) \le m_k(P_n)=\binom{n-k}{k}$. Equality holds if and only if $T=P_n.$
\end{theorem}

\begin{proof}
    We will prove that the statement is true for forests of order $n$ using induction on $k$.

    When $k=2$, for a tree $T$ we have $m_2(T)=\binom{n-1}{2}-\sum_v \binom{\deg(v)}{2}.$ Observe that the degree sequence of the path $P_n$ is majorized by the degree sequence of any other $n$-vertex tree. Indeed, the degree sequence of $P_n$ is $(2, 2, \ldots, 2, 1, 1)$. If $T \neq P_n$ has degree sequence $(d_1 \geq d_2 \geq \cdots \geq d_n)$ and $\sum_{i=1}^j d_i < 2j$ for some $j \leq n-2$, then we must have $d_i = 1$ for all $i \geq j$. Then $\sum_{i=1}^n d_i < 2j + (n-j) \leq 2n-2$, a contradiction since $\sum_{i=1}^n d_i = 2(n-1)$ by the handshaking lemma. Thus, we may apply Lemma~\ref{lem:karamata} (with strict inequality since the degree sequences of $T$ and $P_n$ must be different) to conclude that for any $T \neq P_n$, we have $m_2(T) < m_2(P_n) = {n-2 \choose 2}$.
    
    If $F$ is a forest with at least two components, then $|E(F)| \leq n-2$ and as such, $m_2(F) \leq {n-2 \choose 2}$. Equality is not possible when $n > 4$.

    Now assume the statement holds for $m_{k-1}(T)$ for any $T$ on $n \geq 2k$ vertices.
    When $n=2k+2$, a matching of size $k$ is an almost-perfect matching and so $P_n$ is extremal among trees by Theorem~\ref{thr:even_apm}.
    For a forest $F$ with multiple components, there are three cases.
    
    {\bf Case 1:} All components of $F$ have even size. Let the components of $F$ be $T_1, T_2, \dots, T_r$ where $|T_i| = 2k_i$ for $1 \leq i \leq r$. An almost-perfect matching of $F$ consists of an almost-perfect matching of $T_i$ for some $i$ and a perfect matching of $T_j$ for every $j \neq i$, of which there is at most one by Lemma~\ref{lem:treepm}. Then 
    $$m_k(F) \le \sum_{i=1}^r m_{k_i-1}(T_i) \leq \sum_i m_{k_i - 1}(P_{2k_i}).$$ 
    We claim this is strictly less than $m_k(P_n)$. By Theorem~\ref{thr:even_apm}, it suffices to observe that 
    $$\sum_{i=1}^r \binom{k_i +1}{2} < \binom {\sum k_i + 1}2.$$
    This is true by induction on $r$ since $\binom{a+b+1}{2} > \binom{a+1}{2}+\binom{b+1}2$ for every $a,b \ge 1$ (a straightforward computation).

    {\bf Case 2:} At least four components are odd. Then $m_k(F)=0$ since at least one vertex from each of the four components must be avoided.
    
    {\bf Case 3:} Exactly two components are odd, with order $a$ and $b$. In the last case, we apply Theorem~\ref{thr:odd_apm} and the inequality between the arithmetic and geometric mean (AM-GM) to conclude
    $m_k(F) \le \frac{a+1}{2}\cdot \frac{b+1}{2} \le \left(\frac{k+2}{2}\right)^2 < \binom{k+2}{2}=m_k(P_n).$
    Thus, the statement is true for forests when $n=2k+2.$
    
    So for $n\ge 2k+3$, we may assume that the claim holds for $n-1$.

    Let $F$ be a forest and $u$ a leaf of $F$, with unique neighbor $v$ (a forest without edges is clearly not extremal).
    We observe that $m_k(F)=m_k(F \setminus \{u\}) + m_{k-1}(F \setminus \{u,v\})$, since either $u$ is covered by the matching or it is not.
    By the inductive hypothesis, we have that $m_k(F \setminus \{u\})$ is maximized when $F \setminus u = P_{n-1}$ and $m_{k-1}(F \setminus \{u,v\})$ is maximized when $F \setminus \{u,v\} = P_{n-2}$. This shows that $m_k(F)$ is maximized if and only if $F = P_n$, completing the induction.
\end{proof}

\begin{section}
{Maximum number of strong almost-perfect matchings}\label{sec:sapm}
\end{section}

Recall that a strong almost-perfect matching (\sapm) is a matching which avoids one leaf when $n$ is odd and two leaves when $n$ is even. In this section, we characterize the trees that maximize the number of {\sapm}s with three theorems which, combined, give a more complete version of Theorem~\ref{thr:sapmcombined}.

We first define the relevant families of trees. Often a spider can refer to any subdivision of a star, i.e. any tree which has at most one vertex of degree at least $3$. We will refer to a spider $\Sp_n$ as the unique spider of order $n$ whose ``legs'' all have length $2$, except possibly one leg of length $1$ when $n$ is even.

\begin{defi}
    An {\em odd spider} is a $1$-subdivision of a star. It has odd order. An {\em even spider} is an odd spider with an additional pendant edge added to the unique vertex of degree $\geq 3$, or added to the center of a $P_5$. It has even order.
    We use the notation $\Sp_n$ for a (odd or even depending on parity of $n$) spider of order $n$. See Figure~\ref{fig:spiders}.
    A spider $\Sp_i$ for $i \le 5$ corresponds with a path $P_i$.
\end{defi}

A {\em leg} then refers to any of the 1-subdivided edges or to the additional pendant edge.

\begin{figure}[ht]

\begin{minipage}[b]{.45\linewidth}
\begin{center}
    \begin{tikzpicture}[scale=0.7]
    {
	\foreach \x in {0,1,2,3,4}{\draw[thick] (2,4) -- (\x,2);}				
	\foreach \x in {0,1,2,3,4}{\draw[thick] (\x,2) -- (\x,0);}
	
    \foreach \x in {0,1,2,3,4}{\draw[fill] (\x,2) circle (0.1);}
	\foreach \x in {0,1,2,3,4}{\draw[fill] (\x,0) circle (0.1);}
    \draw[fill] (2,4) circle (0.1);
			
	}
	\end{tikzpicture}\\
\subcaption{An odd spider $\Sp_{11}$}
\label{fig:oddspider}
\end{center}
\end{minipage}\quad\begin{minipage}[b]{.45\linewidth}

\begin{center}
    \begin{tikzpicture}[scale=0.7]
    {
	\foreach \x in {0,1,2,3,4}{\draw[thick] (2,4) -- (\x,2);}				
	\foreach \x in {0,1,2,3,4}{\draw[thick] (\x,2) -- (\x,0);}
	\draw[thick] (2,4) -- (4,4);
    \foreach \x in {0,1,2,3,4}{\draw[fill] (\x,2) circle (0.1);}
	\foreach \x in {0,1,2,3,4}{\draw[fill] (\x,0) circle (0.1);}
    \draw[fill] (2,4) circle (0.1);
	\draw[fill] (4,4) circle (0.1);		
	}
	\end{tikzpicture}\\
\subcaption{An even spider $\Sp_{12}$}
\label{fig:evenspider}
\end{center}
\end{minipage}
\caption{Two spiders}\label{fig:spiders}
\end{figure}
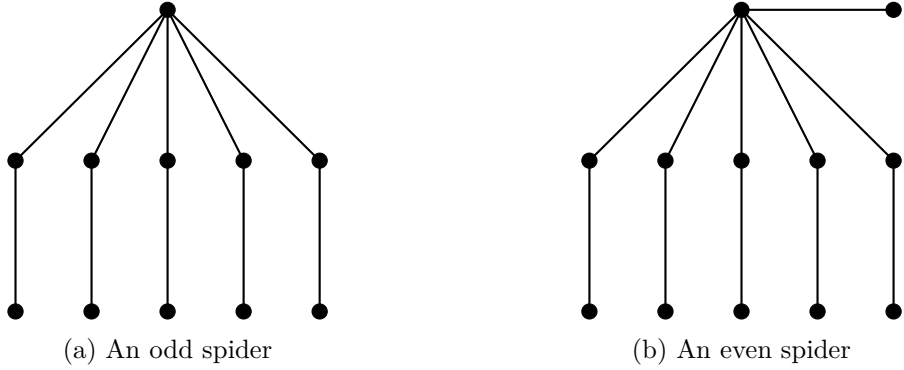

\begin{defi}\label{def:DB}
    A {\em double broom} $\DB_{a,b}$ is a tree of diameter $3$ whose non-leaves have degree $a+1$ and $b+1$ respectively.
    A {\em balanced double broom} on $n$ vertices is a tree of diameter $3$ whose non-leaves have degree $\floorfrac{n}{2} $ on one side and $\ceilfrac{n}{2}$ on the other side. See Figure~\ref{fig:doublebroom}.
   
\end{defi}

\begin{figure}[ht]

\begin{minipage}[b]{.45\linewidth}
\begin{center}
    \begin{tikzpicture}[scale=0.7]
    {
	\foreach \x in {0,1,2,3}{\draw[thick] (0,\x+0.5) -- (2,2);}				
	\foreach \x in {0,1,2,3,4}{\draw[thick] (4,2) -- (6,\x);}
	\draw[thick] (4,2) -- (2,2);
	
    \foreach \x in {0,1,2,3}{\draw[fill] (0,\x+0.5) circle (0.1);}
	\foreach \x in {0,1,2,3,4}{\draw[fill] (6,\x) circle (0.1);}
    \draw[fill] (4,2) circle (0.1);
	\draw[fill] (2,2) circle (0.1);	
	\draw [decorate,decoration={brace,amplitude=4pt},xshift=0pt,yshift=0pt] (-0.25,0.5) -- (-0.25,3.5) node [black,midway,xshift=-0.35cm]{$a$};
	\draw [decorate,decoration={brace,amplitude=4pt},xshift=0pt,yshift=0pt] (6.25,4) -- (6.25,0) node [black,midway,xshift=0.35cm]{$b$};

	}
	\end{tikzpicture}\\
\subcaption{A (balanced) double broom $\DB_{4,5}$}
\label{fig:doublebroom}
\end{center}
\end{minipage}
\quad\begin{minipage}[b]{.45\linewidth}
\begin{center}
    \begin{tikzpicture}[scale=0.7]
    {
	\foreach \x in {0,1,2,3}{\draw[thick] (0,\x+0.5) -- (2,2);}				
	\foreach \x in {0,1,2,3,4}{\draw[thick] (4,2) -- (6,\x);}
	\draw[thick] (4,2) -- (2,2);
	
    \foreach \x in {0,1,2,3}{\draw[fill] (0,\x+0.5) circle (0.1);}
    \foreach \x in {0,1,2,3}{\draw[fill] (1,0.5*\x+1.25) circle (0.1);}
	\foreach \x in {0,1,2,3,4}{\draw[fill] (6,\x) circle (0.1);}
	\foreach \x in {0,1,2,3,4}{\draw[fill] (5,0.5*\x+1) circle (0.1);}
    \draw[fill] (4,2) circle (0.1);
	\draw[fill] (2,2) circle (0.1);		
	}
	\end{tikzpicture}\\
\subcaption{A wide spider $W_{20}$}
\label{fig:widespider}
\end{center}
\end{minipage}
\caption{A double broom and wide spider}\label{fig:db+widespider}
\end{figure}
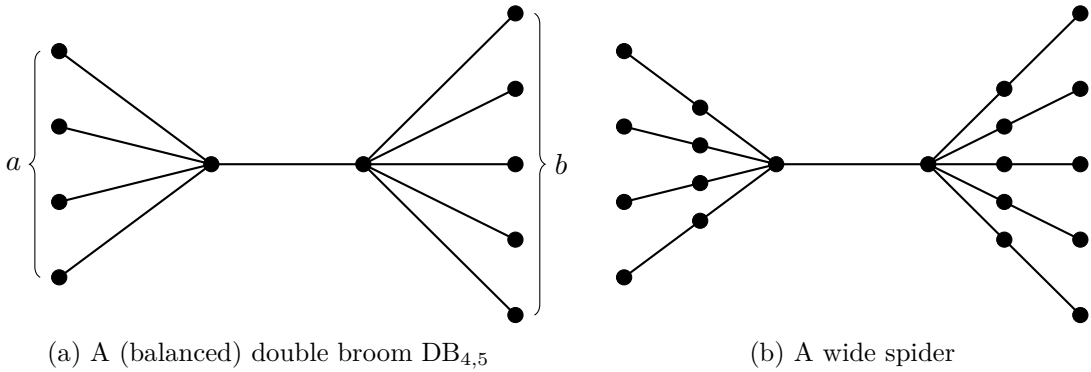

\begin{theorem}\label{thr:odd_sapm}

    Let $n$ be odd. If $n \geq 5$, a tree $T$ of order $n$ has at most $\frac{n-1}{2}$ strong almost-perfect matchings.
    
    For $n \geq 7$, equality holds if and only if $T$ is an odd spider.
    
    For $n=5$, equality holds for both $\DB_{1,2}$ and $P_5.$
\end{theorem}

\begin{proof}
If $T$ is a tree that contains two leaves with a common neighbor, then $T$ has at most two {\sapm}s, since every matching must avoid at least one of these two leaves.
Otherwise, $T$ has at least as many internal vertices as leaves and thus at most $\frac{n-1}{2}$ leaves. It follows that $T$ has at most $\frac{n-1}{2}$ {\sapm}s. The $n$-vertex spider achieves this upper bound.

For $n=5$, we have $\frac{n-1}{2}=2$, and one can check by hand that the extremal graphs are the path $P_5$ (which is a spider) and the double broom $\DB_{1,2}$.
For $n \ge 7$, let $T$ be a tree of order $n$ with $\frac{n-1}{2}$ {\sapm}s. Then $T$ must have $\frac{n-1}{2}$ leaves $\{u_i : 1 \leq i \leq \frac{n-1}{2}\}$. For each $i$, let $v_i$ be the unique neighbor of $u_i$. This gives us $n-1$ vertices, so $T$ contains one additional vertex $x$. For each $i$, $T \setminus \{u_i\}$ has a perfect matching, which must contain $u_jv_j$ for all $j \neq i$. Thus, $x \sim v_i$ for all $i$, and so $T$ is indeed a spider.
\end{proof}

When $n$ is even, an almost-perfect matching $M$ of $T$ contains all but two vertices, so counting the number of almost-perfect matchings of $T$ is equivalent to counting the number of avoided pairs $P \subset V(T)$ for which $T \setminus P$ has a perfect matching. This is the strategy we employ to prove the next theorem, in which we have the additional assumption that $T$ contains a perfect matching. 

    \begin{defi}\label{def:wide-spider}
A {\em wide spider} is constructed by subdividing every pendant edge of a double broom. For $n=2k$, the {\em balanced wide spider} $W_n$ on $n$ vertices is the tree one obtains from subdividing every pendant edge of a balanced double broom on $k+1$ vertices. See Figure~\ref{fig:widespider}.
\end{defi}

\newpage

\begin{theorem}\label{thr:even_sapm_wpm}
Let $n$ be even. If $T$ is a tree of order $n$ that contains a perfect matching, then the number of strong almost-perfect matchings in $T$ is at most
%The greatest possible number of {\sapm}s in a tree of order $n$ with a perfect matching is
$$\max \left\{1,\frac{n-2}{2}, \floorfrac{(n-2)^2}{16}\right\} = \begin{cases} 1 & n = 2, \\ \frac{n-2}{2} & 2 < n \leq 10, \\ \floorfrac{(n-2)^2}{16} & n \geq 10.\end{cases}$$
Equality holds if and only if either
\begin{itemize}
    \item $n \leq 10$, and $T$ is obtained by attaching a leaf to each vertex of a tree of order $\frac{n}{2}$, or
    \item $n \geq 10$, and $T$ is the balanced wide spider.
\end{itemize}
\end{theorem}

\begin{proof}
If $n = 2$, then there is precisely one \sapm, namely the empty matching, so assume that $n > 2$.
Let $T$ be an arbitrary tree of order $n$ with a perfect matching. Let $\{u_1, \dots, u_k\}$ be the leaves of $T$. Observe that no two leaves have the same neighbor; else, the perfect matching would match both with their common neighbor. It follows that $k \leq \frac{n}{2}$. Let $v_i$ be the unique neighbor of $u_i$, $1 \leq i \leq k$.

Recall that a bipartite graph can have a perfect matching only if the parts of the bipartition are the same size. Consider the (unique) bipartition $V(T) = X \cup Y$. An \sapm{} of $T$ that avoids $u_i$ and $u_j$ is a perfect matching of $T \setminus \{u_i, u_j\}$. This can occur only if $u_i$ and $u_j$ belong to different parts of the bipartition $X \cup Y$. Thus, there are at most $\floorfrac{k^2}{4}$ {\sapm}s of $T$.

If $k = \frac{n}{2}$, then $T[\{v_1, \dots, v_k\}]$ must be a tree of order $\frac{n}{2}$, and every \sapm{} in $T$ can only avoid $u_i$ and $u_j$ such that $v_i \sim v_j$. This gives $\frac{n}{2}-1 = \frac{n-2}{2}$ {\sapm}s in $T$, one for each edge of $T[\{v_1, \dots, v_k\}]$.

Otherwise, $k \leq \frac{n}{2} - 1 = \frac{n-2}{2}$, which means that there are at most $\floorfrac{(n-2)^2}{16}$ {\sapm}s. This is achieved by the wide spider $W_n$. Indeed, for every pair of leaves on opposite sides, there is an \apm{} that avoids those leaves, so the number of {\sapm}s in $W_n$ is $\floorfrac{(n-2)^2}{16}$.

Thus, we have shown that the maximum is
$$\max \left\{1,\frac{n-2}{2}, \floorfrac{(n-2)^2}{16} \right\},$$
and it remains to discuss the cases of equality. Note first that
$\frac{n-2}{2} \geq \floorfrac{(n-2)^2}{16}$ if and only if $n \leq 10$ (with equality for $n=10$), so the case that there are $k = \frac{n}{2}$ leaves yields the maximum for $2 < n \leq 10$, while the case that there are $k = \frac{n}{2} - 1$ leaves and $\floorfrac{k^2}{4}$ {\sapm}s yields the maximum for $n \geq 10$. 

It only remains to show that the wide spider $W_n$ is unique in the second case. Suppose $T$ has $k = \frac{n-2}2$ leaves and $\floorfrac{(n-2)^2}{16}$ almost-perfect matchings. Let $x$ and $y$ be the two vertices that are neither leaves nor neighbors of leaves. Their degree is at least $2$. The perfect matching of $T$ must contain the edge $u_iv_i$ for all $i$, so $x \sim y$. Suppose without loss of generality that $x \in X, y \in Y$. Then there is some $v_i \in X, v_j \in Y$ such that $x \sim v_j$ and $y \sim v_i$.

Let $v_\ell \in X$ for $\ell \neq i$. Suppose for the sake of contradiction that $v_\ell \not\sim y$. For each $m$ such that $v_m \in Y$, in order to have a perfect matching of $T \setminus \{u_\ell, u_m\}$, we must have $v_\ell \sim v_m$. This implies $v_m \not\sim x$ for all $m \neq j$; else, $\{x, v_j, v_\ell, v_m\}$ would form a cycle. Then there cannot be an almost-perfect matching that avoids $u_i$ and $u_m$ for $m \neq j$, since there is no disjoint set of edges to cover $v_m, v_\ell$, and $u_\ell$, which is a contradiction.
This tells us that $v_\ell \sim y$ for all $\ell$ such that $v_\ell \in X$, and a similar argument shows that $v_m \sim x$ for all $m$ such that $v_m \in Y$, and so $T$ is isomorphic to the balanced wide spider.
\end{proof}

When we remove the requirement that $T$ contains a perfect matching, we can still characterize the extremal examples as follows:

\begin{defi}
An {\em even special spider} $\SSS_n$ on $n$ vertices for $n$ even is an odd spider $\Sp_{n-1}$ with a sibling added for one of the leaves. See Figure~\ref{fig:evenspecialspider}.

A {\em spider trio} $\ST_{a,b,c}$ is a tree constructed from three odd spiders $\Sp_{2a+1}$, $\Sp_{2b+1}$, $\Sp_{2c+1}$ whose center vertices are attached to a single new vertex. See Figure~\ref{fig:spider-trio}.
\end{defi}

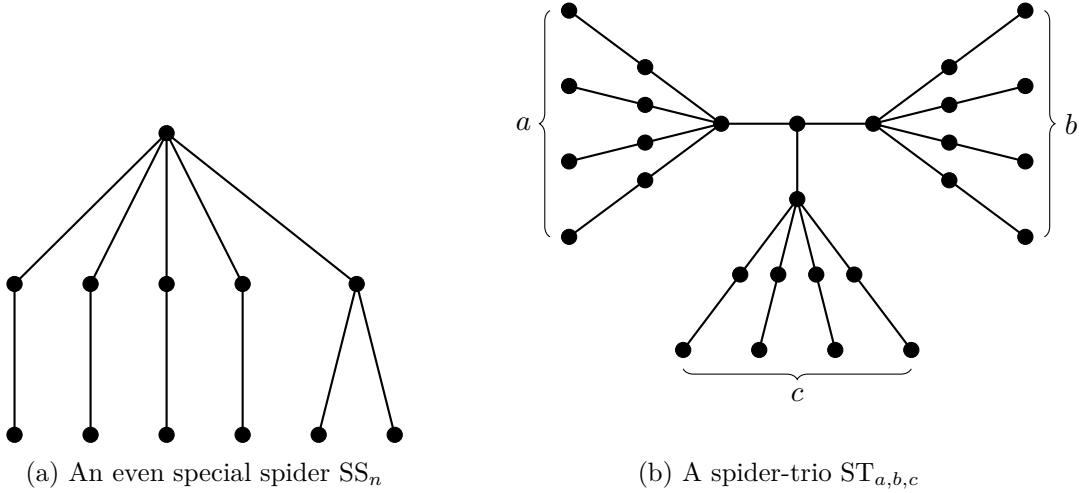
\begin{figure}[ht]
\begin{minipage}[b]{.45\linewidth}
\begin{center}
    \begin{tikzpicture}[scale=0.7]
    {
	\foreach \x in {0,1,2,3,4.5}{\draw[thick] (2,4) -- (\x,2);}				
	\foreach \x in {0,1,2,3}{\draw[thick] (\x,2) -- (\x,0);}
	
    \foreach \x in {0,1,2,3,4.5}{\draw[fill] (\x,2) circle (0.1);}
	\foreach \x in {0,1,2,3,4,5}{\draw[fill] (\x,0) circle (0.1);}
    \draw[fill] (2,4) circle (0.1);
    
    \draw[thick] (4.5,2) -- (4,0);
    
    \draw[thick] (4.5,2) -- (5,0);
	}
	\end{tikzpicture}\\
\subcaption{An even special spider $\SSS_{12}$}
\label{fig:evenspecialspider}
\end{center}
\end{minipage}\quad\begin{minipage}[b]{.45\linewidth}

		\begin{center}
			\begin{tikzpicture}[scale=0.7]
			{
				\foreach \x in {0,1,2,3}{\draw[thick] (0,\x+0.5) -- (2,2);}				
				\foreach \x in {0,1,2,3}{\draw[thick] (4,2) -- (6,\x+0.5);}
				\draw[thick] (4,2) -- (2,2);
				\draw[thick] (3,2) -- (3,1);
				
				\foreach \x in {0,1,2,3}{\draw[fill] (0,\x+0.5) circle (0.1);}
				\foreach \x in {0,1,2,3}{\draw[fill] (1,0.5*\x+1.25) circle (0.1);}
				\foreach \x in {0,1,2,3}{\draw[fill] (6,\x+0.5) circle (0.1);}
				\foreach \x in {0,1,2,3}{\draw[fill] (5,0.5*\x+1.25) circle (0.1);}
				\draw[fill] (4,2) circle (0.1);
				\draw[fill] (2,2) circle (0.1);	
				\draw[fill] (3,2) circle (0.1);
				\draw[fill] (3,1) circle (0.1);
				
				\foreach \x in {0,1,2,3}{\draw[fill] (\x+1.5,-1) circle (0.1);}
				\foreach \x in {0,1,2,3}{\draw[fill] (0.5*\x+2.25,0) circle (0.1);}
				\foreach \x in {0,1,2,3}{\draw[thick] (\x+1.5,-1) -- (3,1);}		
				
				\draw [decorate,decoration={brace,amplitude=4pt},xshift=0pt,yshift=0pt] (-0.25,0.5) -- (-0.25,3.5) node [black,midway,xshift=-0.35cm]{$a$};
				\draw [decorate,decoration={brace,amplitude=4pt},xshift=0pt,yshift=0pt] (6.25,3.5) -- (6.25,0.5) node [black,midway,xshift=0.35cm]{$b$};	
				\draw [decorate,decoration={brace,amplitude=4pt},xshift=0pt,yshift=0pt] (4.5,-1.25)--(1.5,-1.25)  node [black,midway,yshift=-0.35cm]{$c$};

			}
			\end{tikzpicture}
		\end{center}
	\subcaption{A spider-trio $\ST_{4,4,4}$}\label{fig:spider-trio}
%\end{center}
\end{minipage}
\caption{Other variants of spiders}\label{fig:adaptedspiders}
\end{figure}

\begin{theorem}\label{thr:even_sapm_gen}
Let $f(n)$ be the maximum number of {\sapm}s in a tree $T$ of order $n$. Then, for $n$ even,
	$$f(n)=
\begin{cases}
\floorfrac{3n}{4} & \text{ if } 2 \le n \le 6,\\
n-3 & \text{ if } 8 \le n \le 14,\\
 \floorfrac{n^2}{16}-1 & \text{ if } 16 \le n \le 28,\\
\floorfrac{(n-4)^2}{12}  & \text{ if } n \ge 30.\\
\end{cases}$$
Equality holds if and only if $T$ is one of the following trees:
$$\begin{cases}
P_2, S_4, \textrm{or } \DB_{2,2} &\text{ for } 2 \le n \le 6,\\
\SSS_n & \text{ if } 8 \le n \le 12,\\
\SSS_{14} \textrm{ or } \ST_{3,2,0} & \text{ if } n=14,\\
\ST_{\ceilfrac{n-4}{4}, \floorfrac{n-4}{4},0}  & \text{ if } 16 \le n \le 26,\\
\ST_{6,6,0} \textrm{ or } \ST_{4,4,4}   & \text{ if } n=28,\\
\ST_{\floorfrac{n}{6}, \floorfrac{n-2}{6},\floorfrac{n-4}{6}}  & \text{ if } n \ge 30.\\
\end{cases}$$
\end{theorem}

\begin{proof}
    For $n \leq 16$, the values of $f(n)$ have been checked with a computer, so we may assume $n\ge 18$.
	Let $T$ be a tree on $n$ vertices. If two leaves of $T$ share a common neighbor, then by Lemma~\ref{lem:leavesatdistance2=>n_bound}, $T$ has at most $n-1$ {\sapm}s. Since $f(n)>n$ (as evidenced by the trees given in the claim), $T$ cannot be extremal in this case. 
	
	Thus, we may also assume that every leaf has a unique neighbor and so the number of leaves is at most $\frac n2$.
	If $T$ has exactly $\frac n2$ leaves, then $T$ must also contain a perfect matching. By Theorem~\ref{thr:even_sapm_wpm}, $T$ has at most $\floorfrac{(n-2)^2}{16}$ {\sapm}s, which is smaller than both $\floorfrac{n^2}{16}-1$ and $\floorfrac{(n-4)^2}{12}$ when $n \geq 18$.
	
	Now suppose the number of leaves in $T$ is at most $\frac n2-1.$ Let $V(T)=X \cup Y$ be the bipartition of $T$, and without loss of generality, suppose $\abs X \geq \abs Y$. Recall that $T$ containing an almost-perfect matching implies that either $\abs X = \abs Y$ or $\abs X = \abs Y + 2$.
	If $\abs X = \abs Y,$ then every avoided pair of leaves must have one element in $X$ and one element in $Y$. There are at most $\floorfrac{(n-2)^2}{16}$ such pairs, which again is less than the stated upper bounds.
	
	So now we assume that $\abs X = \abs Y+2=\frac n2+1.$
	
	{\bf Case 1:} $Y$ contains no leaves. Then the degree of every vertex in $Y$ is at least $2$. As $\abs Y = \frac{n-2}{2}$, and $\sum_{v \in Y}\deg(v) = n-1$, this implies that all but one vertex in $Y$ has degree 2. Let $z$ be the unique vertex in $Y$ which has degree 3.
	
	In that case $T \setminus \{z\}$ has three components---call them $A, B, C$. Each component is itself a tree whose vertices in $Y$ have degree 2. Thus each of them has one more vertex in $X$ than in $Y$. In particular, each component has an odd number of vertices in total. 
	
	Observe that if a component, say $C$, of $T \setminus \{z\}$ does not contain a member of the avoided pair, then, since $\abs C$ is odd, the \sapm{} must match $z$ with its neighbor in $C$. Thus, an avoided pair of leaves cannot be contained in a single component.

	Let $a,b,c$ be the number of leaves in $A, B$, and $C$ respectively, with $a \ge b \ge c$. Then the number of {\sapm}s is bounded above by $ab+bc+ac.$
	If $B$ and $C$ are singletons, this gives an upper bound of $1+2a< n$.
	If only $C$ is a singleton, then we have $ab+a+b=(a+1)(b+1)-1\le \floorfrac{n^2}{16}-1.$
	Furthermore, we can see that equality is attained only for the tree $\ST_{a,b,0}$ when $a = b = \frac{n}{4}$. Indeed, each \sapm{} of $T$ is a union of {\sapm}s of $A$ and $B$, and the edge from $z$ to $C$. By Theorem~\ref{thr:odd_sapm}, $A$ and $B$ have the maximum number of {\sapm}s if and only if they are both spiders.
	
	Else, each neighbor of $z$ is a non-leaf vertex in $X$, and as such $a+b+c \le \frac{n-4}{2}.$
	By the AM-GM inequality, we know that $ab+bc+ac \le \frac{(a+b+c)^2}{3}$ and since the number of {\sapm}s is an integer, we conclude that $\floorfrac{(n-4)^2}{12}$ is an upper bound.
	By an argument similar to the previous case, the extremal tree is precisely the balanced spider-trio $\ST_{a,b,c}$ where $a\ge b \ge c \ge a-1$. This has exactly $\floorfrac{(n-4)^2}{12}$ {\sapm}s.
	It is now sufficient to note that $\floorfrac{n^2}{16}-1 \le \floorfrac{(n-4)^2}{12} $ if and only if $n \ge 28.$
	
	{\bf Case 2:} There is a leaf $y\in Y$. Let $x$ be the unique neighbor of $Y$.
	Since $\abs X = \abs Y+2$, every avoided pair must be contained in $X$ and so every almost-perfect matching must contain $xy$.
	
	If $\deg(x)=2$, the number of {\sapm}s of $T$ equals the number of {\sapm}s of $T\setminus \{x, y\}$, which is at most $f(n-2)$. It is straightforward to verify algebraically that $ \floorfrac {n^2} {16} - 1$ (for $16 \leq n \leq 28$) and $\floorfrac {(n - 4)^2} {12}$ (for $n \geq 28$) are increasing functions, so the result follows by induction.
	
	If $\deg(x)>2$ and $T\setminus \{x, y\}$ is a forest, the result also follows similarly by induction since $T\setminus \{x, y\}$ is a subgraph of a tree.
	If the components of $T\setminus \{x, y\}$ are all even, with orders $n_1, n_2, \ldots n_k$, then there is some $i \in \{1, \dots, k\}$ such that the component of order $n_i$ has two more vertices in $X$ than in $Y$. This implies that $T$ has at most $f(n_i) < f(n)$ {\sapm}s.
%    are at most $f(n_1)+f(n_2)+\ldots+f(n_k)<f(n)$ {\sapm}s.
	If two components are odd, we have at most $\floorfrac{(n-2)^2}{16}$ pairs of leaves, one in each odd component. Therefore, there are at most $\floorfrac{(n-2)^{2}}{16} \leq\floorfrac{(n-4)^{2}}{12}$ SAPMs.
\end{proof}

\begin{section}
{Minimum number of maximal matchings}\label{sec:minmaxmatching}
\end{section}

In this section, we characterize trees with the minimum number of maximal matchings. This complements the result of G\'{o}rska and Skupie\'{n}~\cite{GS07} on the maximum number of maximal matchings.

\begin{defi}
An {\em odd special spider} $\SSS_n$ on $n$ vertices for $n$ odd is an even spider $\Sp_{n-1}$ with a pendant edge added to the center vertex (for $n \geq 7$, this is the unique vertex of degree $\geq 3$). See Figure~\ref{fig:oddSSSn}.
\end{defi}

\begin{figure}[ht]
    \centering
    \begin{tikzpicture}[scale=0.7]
    {
	\foreach \x in {0,1,2,3,4}{\draw[thick] (2,4) -- (\x,2);}				
	\foreach \x in {0,1,2,3,4}{\draw[thick] (\x,2) -- (\x,0);}
	
    \foreach \x in {0,1,2,3,4}{\draw[fill] (\x,2) circle (0.1);}
	\foreach \x in {0,1,2,3,4}{\draw[fill] (\x,0) circle (0.1);}
     \foreach \x in {0,2,4}{
     \draw[fill] (\x,4) circle (0.1);
     }
    
    \draw[thick] (2,4) -- (4,4);
    \draw[thick] (2,4) -- (0,4);
	}
	\end{tikzpicture}
    \caption{An odd special spider $\SSS_n$ where $n=13$}
    \label{fig:oddSSSn}
\end{figure}
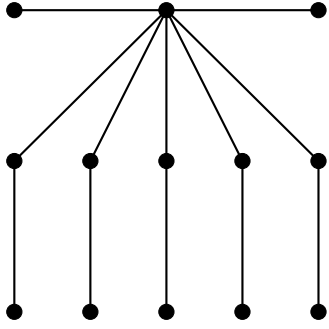

\begin{customthm} {\bf \ref{thr:minmaxmatchings}}
    A tree $T$ of order $n$ has at least $\lceil \frac n2 \rceil $ maximal matchings. 
    Equality holds if and only if $T$ is a spider $\Sp_n$, or an odd special spider $\SSS_n$.
\end{customthm}

\begin{proof}
    We proceed by induction on $n$. A small computer verification confirms the claim for $n\le 7$ (the paths $P_n$ are also spiders $\Sp_n$ when $n \le 5$).
    Now let $n \ge 8$. 
    We consider two cases. First, suppose there are two leaves $u_1, u_2$ with a common neighbor $v$. 
    %or there is a leaf $u$ whose unique neighbor $v$ has degree $2$, i.e., has only $2$ neighbors $\ell_1$ and $u.$
    Let $T' = T\setminus \{u_1, u_2\}$. Observe that every maximal matching $M$ of $T'$ can be extended to a maximal matching of $T$. If $M$ covers $v$, then $M$ is also a maximal matching of $T$. If $M$ does not cover $v$, then we can extend $M$ with $u_1v$ or with $u_2v$. Thus, $T$ has at least as many maximal matchings as $T'$.
    
    In the second case, no two leaves share a common neighbor. Then there must exist a leaf $u$ whose unique neighbor $v$ has degree 2 (for example, suppose $T$ is rooted at an arbitrary vertex $\rho$ and take $u$ to be a leaf at maximum depth).
    Let $T' = T \setminus \{u, v\}$ and let $x$ be the other neighbor of $v$. As before, every maximal matching $M$ of $T'$ can be extended to a maximal matching of $T$ and so $T$ has at least as many maximal matchings as $T'$. Indeed, if $M$ covers $x$, the unique extension is given by $M \cup \{uv\}$. If $M$ does not cover $x$, we can extend using either $uv$ or $vx$. 

    Moreover, in both cases, we can always find a maximal matching of $T$ that uses a particular edge not contained in $T'$ (either $u_1v$ or $u_2v$ in the first case, and either $uv$ or $vx$ in the second case). This shows that $T$ has strictly more maximal matchings than $T'$, and we conclude by induction.

    For equality to be attained by $T$, equality needs to be attained by $T'$ and thus $T'$ is a spider or an odd special spider. Note that $T'$ has at least two legs of length $2$ and thus $T$ has at least one leg of length $2$, say with leaf vertex $y$ and neighbor $z$. By applying the inductive hypothesis to $T \setminus \{y,z\}$, we know that $T \setminus \{y,z\}$ has to be either $\Sp_{n-2}$, or $\SSS_{n-2}$ if $n$ is odd. We can then conclude that $T$ itself is $\Sp_{n}$ or $\SSS_{n}$, and these indeed attain equality.
\end{proof}

Initially, this question about determining the minimum number of maximal matchings was motivated by the following observation about maximal matchings, which we will utilize in applications to the weighted Hosoya index. We say the {\em degree sum} of a maximal matching $M$ to mean $\sum_{v \in V(M)} \deg(v)$. 

\begin{lemma}\label{lem:sumdeg_maxmatching}
If $n \geq 2$, then the degree sum of a maximal matching $M$ in a tree $T$ of order $n$ is at least $n$.
Equality occurs only if $T$ is a star, or $T$ has diameter $3$ and $M$ is the central edge.
\end{lemma}
Note that the statement is trivially false for a $1$-vertex tree, so $n \geq 2$ is required.
\begin{proof}
    Let $M$ be a maximal matching. If $v \in V$ is not covered by $M$, then there must exist some $u \in N(v)$ that is covered by $M$; else, $M \cup \{uv\}$ is a larger matching, contradicting the maximality of $M$. This implies that $\bigcup\{N(u) : u \text{ is covered by }M\} = V$, from which the statement follows.
    
    Equality can occur only if every vertex has precisely one neighbor covered by the matching. By the argument above, for every $v$ not covered by $M$, we must have all of $N(v)$ covered by $M$. This implies that all non-leaf vertices of $T$ are covered by $M$. If the diameter of $T$ is at least $4$, then there exist internal vertices of a diameter path which have more than one neighbor covered by $M$, and so equality does not occur. If $T$ has diameter $3$, there are only two types of maximal matchings and we conclude easily.
\end{proof}

As a corollary of Theorem~\ref{thr:minmaxmatchings} and Lemma~\ref{lem:sumdeg_maxmatching}, we have the following result on the weighted Hosoya index. This is a special case of a more general theorem that we prove in the next section. Recall that $S_n$ denotes the $n$-vertex star.
 
\begin{corollary}\label{exponential}
    Let $\phi(i,j) = c^{i+j}$ where $c \geq 2$. Then for every $n$-vertex tree $T$, we have $Z_\phi(S_n) \leq Z_\phi(T)$.
\end{corollary}

\begin{proof}
    First note that $S_n$ has $n-1$ nonempty matchings, each with weight $c^n$, so $Z_{\phi}(S_n)=(n-1)c^n+1$.
    For a tree $T$ that is not a star, by Theorem~\ref{thr:minmaxmatchings}, $T$ has at least $\lceil \frac{n}{2} \rceil$ maximal matchings. By Lemma~\ref{lem:sumdeg_maxmatching},  each maximal matching has degree sum strictly greater than $n$, with possibly one exception which has degree sum exactly equal to $n$. Thus we have 
    $Z_{\phi}(T)> \left(\ceilfrac{n}{2}-1\right) c^{n+1}+c^n+1 \ge (n-1)c^n+1=Z_{\phi}(S_n)$.
    The first inequality is strict since there are pendant edges which are not maximal matchings and contribute in the sum as well. The second inequality holds by our assumption $c \geq 2$.
\end{proof}

\begin{section}
{Extremal values of the weighted Hosoya index}\label{sec:HosoyaIndex}
\end{section}

Recall that the weighted Hosoya index is defined to be
$$Z(T, \omega) = \sum_{M \in \mathcal{M}(T)} \prod_{e \in M} \omega(e)$$
where the weight of the empty matching is $1$.
In this section, we characterize the extremal examples for some natural choices of weights that are vertex-degree-based, meaning $\omega(e) = \phi(i,j)$ where $i$ and $j$ are the degrees of the endpoints of $e$. We then write $Z(T, \omega)$ as $Z_\phi(T)$. 

We consider the following weight functions for $c > 0$:
\begin{align*} \phi_1(i,j) &= c^{i+j},\\
\phi_2(i,j) &= c^{ij},\\
\phi_3(i,j) &= (i+j)^c,\\
\phi_4(i,j) &= (ij)^c.
\end{align*}

Note that these examples are all symmetric in $i$ and $j$. It was shown in~\cite{CRUZ2022} that for $c < 0$, the star minimizes the weighted Hosoya index for $\phi_3$ and $\phi_4$ (also referred to as the general Randi\'{c} index and general sum-connectivity index, respectively). 
This can also be derived from the following proof that the star minimizes the weighted Hosoya index for $c>0$.

\begin{theorem}\label{thr:starminZ_phi}
Let $c > 0$. Then for every $n$-vertex tree $T$, we have $Z_{\phi_\ell}(T) \geq Z_{\phi_\ell}(S_n)$ for all $\ell \in \{1,2,3,4\}$. Equality holds in each case if and only if $T = S_n$.
\end{theorem}

\begin{proof}
Let $Z_\phi(T, v, k)$ be the weighted Hosoya index of $T$ where the weight of $v$ is increased by $k$. That is, if $uv = e \in E(T)$ for some $u$, then $\omega(e) = \phi(\deg(u), \deg(v)+k)$. We will use induction to prove a stronger statement:
\begin{claim}
Let $T$ be a tree on $n$ vertices, and let $v \in V(T)$ be a non-leaf vertex that has at least one leaf neighbor. For any nonnegative integer $k$ and $\ell \in \{1,2,3,4\}$, $Z_{\phi_\ell}(T, v, k) \geq Z_{\phi_\ell}(S_n, x, k)$ where $x$ is the center vertex of the star. 
\end{claim}
Observe that taking $k = 0$ gives the theorem statement.

We proceed by induction on $n$. We can check the first non-trivial base case ($n=4$) directly to see that the claim holds true for all $k \geq 0$ and all $c > 0$. 

Now suppose that for every $n' < n$, the claim holds for all $k \geq 0$.

Let $T$ be an $n$-vertex tree different from $S_n$, and let $k$ be a nonnegative integer.
Let $v \in T$ be a non-leaf vertex. Let $u_1, \dots, u_s$ be the $s \ge 1$
leaf neighbors of $v$ and let $x_1, \dots, x_t$ be the non-leaf neighbors.
Let $T_i$ be the component of $x_i$ in $T \setminus \{v, u_1, \dots, u_s\}$.
Then
\begin{equation}\label{eqn:hosoya-cpts}
Z_\phi(T, v, k) = Z_\phi(T\setminus \{u_1, \dots,u_s\}, v, s+k) + s \phi(1, \deg(v)+k)\prod_{i=1}^t Z_\phi(T_i, x_i, 1).
\end{equation}
Given a star graph $S_n$, let $z_n$ denote the center vertex.

First consider $\phi_1(i,j) = c^{i+j}$. We can calculate that $Z_{\phi_1}(S_n, z_n, k) = (n-1)c^{n+k}+1$. By \eqref{eqn:hosoya-cpts} and the inductive hypothesis,
\begin{align*} Z_{\phi_1}(T, v, k) &\geq Z_{\phi_1}(S_{n-s},z_{n-s}, s+k) + s\phi_1(1, s+t+k)\prod_{i=1}^t Z_{\phi_1}(S_{|T_i|},z_{|T_i|}, 1)\\
&\geq (n-s-1)c^{(n-s)+(s+k)}+1+sc^{(s+t)+k+1}\prod_{i=1}^t ((|T_i|-1)c^{|T_i|+1}+1).
\end{align*}
Observe that $\sum_{i=1}^t |T_i| = n-s-1$, so we have
\begin{align*}
Z_{\phi_1}(T, v, k) &\geq (n-s-1)c^{n+k}+1+sc^{s+t+k+1}\left(c^{(n-s-1)+t}\prod_{i=1}^t (|T_i|-1)+1\right)\\
&= (n-s-1)c^{n+k}+1+sc^{n+2t+k}\prod_{i=1}^t (|T_i|-1)+sc^{s+t+k+1}.\\
\end{align*}

We want the last expression to be at least $(n-1)c^{n+k}+1$. Recall that $s \geq 1$ by assumption. Dividing both sides of the desired inequality by $sc^{n+k}$ and rearranging, it suffices to show
$$c^{2t} \prod_{i=1}^t (|T_i|-1) +c^{s+t+1-n}\geq 1.$$
If $T$ is not isomorphic to $S_n$, then $t \geq 1$ and $|T_i| \geq 2$ for all $i$, so when $c \geq 1$, the first term is at least 1 and the inequality holds. The second term is also positive, so the inequality must be strict in this case. 

When $c < 1$, we use that $n = (s+1)+\sum_{i=1}^t |T_i| \geq s+1+2t$, so the second term is at least $1$, and as before, the inequality holds and is strict for $T \neq S_n$. 

(As a remark, note that a more direct argument works for $\phi_1$ in the case $c < 1$. Since $T$ is a tree, $\deg(u)+\deg(v) \leq n$ for every edge $uv$ in $T$. Thus, $Z_{\phi_1}(T) \geq (n-1)c^n + 1$ where equality holds if and only if $T$ is the star. However, this argument does not apply to the case of, for example, $\phi_2(i,j) = c^{ij}$.)

The calculations are almost identical for $\phi_2, \phi_3$, and $\phi_4$, so we omit many of the details here. Each computation begins by applying \eqref{eqn:hosoya-cpts} and the inductive hypothesis.

For $\phi_2$, we have $Z_{\phi_2}(S_n, z_n, k) = (n-1)c^{n+k-1}+1$ whereas for general $T$, we have
$$Z_{\phi_2}(T, v, k) \geq (n-s-1)c^{n+k-1}+1+sc^{n+t+k-1}\prod (|T_i|-1) + sc^{s+t+k}$$
and similarly to the case of $\phi_1$, we can show that this is at least $(n-1)c^{n+k-1}+1$ with equality if and only if $8T = S_n$. 

For $\phi_3$, we have $Z_{\phi_3}(S_n, z_n, k) = (n-1)(n+k)^c + 1$ and on the other hand
$$Z_{\phi_3}(T, v, k) \geq (n-s-1) (n+k)^c + 1 + s (s+t+k+1)^c \prod (|T_i|+1)^c .$$
Since $n+k < n+k+t= (s+t+k)+\sum |T_i|+1$ and each of the terms $(s+t+k)$ and $\sum |T_i|$ is at least $2$,
we know that $n+k \le (s+t+k+1)\prod (|T_i|+1)$ and so 
$$Z_{\phi_3}(T,v,k) \geq (n-s-1)(n+k)^c + 1 + s(n+k)^c = (n-1)(n+k)^c + 1$$ as desired.

Lastly, for $\phi_4$, we have $Z_{\phi_4}(S_n, z_n, k) = (n-1)(n+k-1)^c+1$ and
$$Z_{\phi_4}(T, v, k) \geq (n-s-1) (n+k-1)^c + 1 + s (s+t+k)^c \prod (|T_i|)^c.$$
Similar to the previous case, the result follows from 
$n+k-1 \le (s+t+k)\prod |T_i|$.
\end{proof}

From this and previous work on the Hosoya index, one might be tempted to conjecture that the path and the star are always extremal, even if not uniquely so. It turns out this is not the case. Recall the definitions of the double-broom $\DB_{a,b}$ (\ref{def:DB}) and the wide spider $W_n$ (\ref{def:wide-spider}).

\begin{theorem}\label{thr:exponential_max}
For $c$ sufficiently large in terms of $n$, the tree that (uniquely) maximizes $Z_{\phi_\ell}(T)$ among all trees of order $n$ is
\begin{itemize}
    \item\label{thm:exp_max1}$\ell=1$: $W_n$ if $n\ge 12$ is even and $\Sp_n$ if $n \geq 7$ is odd,
    \item\label{thm:exp_max2}$\ell = 2$: $\DB_{\floorfrac{n}{2}-1,\ceilfrac n2 -1}$ if $n \ge 6,$
    \item\label{thm:exp_max34}$\ell \in \{3,4\}$: $P_n$. 
\end{itemize}
\end{theorem}

\begin{proof}[Proof for {\hyperref[thm:exp_max1]{$\phi_1$}}]
    Recall that $\phi_1(i,j) = c^{i+j}$. Suppose first that $n$ is even.
    For a perfect matching $M$, we must have $$\sum_{e=uv \in M} \left( \deg(u)+\deg(v) \right) =2(n-1).$$
    Thus, for $c \geq 1$, the leading term of $Z_{\phi_1}(T)$ is at most $c^{2n-2}$ (since every tree has at most one perfect matching). The second leading term corresponds to the weight of strong almost-perfect matchings in $T$, which must each have weight $c^{2n-4}$. Because we are taking $c$ sufficiently large, it is enough to recall from Theorem~\ref{thr:even_sapm_wpm} that the balanced wide spider uniquely maximizes the number of {\sapm}s for $n \geq 12$.
    
    If $n$ is odd, the leading term of $Z_{\phi_1}(T)$ is given by {\sapm}s which each have weight $c^{2n-3}$. From Theorem~\ref{thr:odd_sapm}, we know that the odd spider has the maximum number of {\sapm}s when $n \geq 7$.
    \end{proof}
    
    \begin{proof}[Proof for {\hyperref[thm:exp_max2]{$\phi_2$}}]
    To prove the result for $\phi_2(i,j) = c^{ij}$ we use the following claim.
    \begin{claim}
    Let $T$ be a tree of order $n \ge 6$ and $M$ a matching of $T$.
    Then $\sum_{uv \in M} \deg(u)\cdot \deg(v) \le \floorfrac{n^2}{4}$, with equality if and only if $T$ is the balanced double-broom $\DB_{\floorfrac{n}{2}-1,\ceilfrac{n}{2}-1}$ with diameter $3$ and $M$ consists of the central edge.
    \end{claim}
    \begin{claimproof}[Proof of Claim]
        Let $M = \{e_1,\dots, e_k\}$ and for every $i$, let $d_i$ be the sum of the degrees of the endpoints of $e_i$. Assume the edges are ordered such that $d_1 \leq d_2 \leq \cdots \leq d_k$.
        Since $n>2$, we know $d_i\ge 3$ for every $i$.
        There are $n-2k$ vertices not covered by $M$, so by the handshaking lemma, we know that 
        $\sum_{i=1}^k d_i \le 2(n-1)-(n-2k)=n+2k-2.$
        If $k=1$, then $M$ consists of a single edge $uv$ and by the AM-GM inequality,
        $$\deg(u)\deg(v) \le \floorfrac{d_1^2}{4} \le \floorfrac{n^2}{4}.$$
       
        When $k \ge 2$, it suffices to prove that 
        $$\sum_i \floorfrac{d_i^2}{4} \le \floorfrac{n^2}{4}.$$ 
        In this case, we apply Lemma~\ref{lem:karamata}. Consider the sequences 
        $A= (3,3,\dots, 3, \sum_{i=1}^k d_i - 3(k-1))$ with the first $k-1$ terms equal to 3, and $B = (d_1, \dots, d_k)$. Both sequences sum to $\sum_{i=1}^k d_i$ and since $d_i \geq 3$ for all $i$, we have that for any $\ell < k$, the sum of the first $\ell$ terms of $A$ is at most the sum of the first $\ell$ terms of $B$. This means that $A$ majorizes $B$, so we may apply Karamata's inequality with the convex function $f(x) = x^2$
        to get that
        $\sum_{i=1}^k d_i^2 \leq (k-1)3^2+ ((n+2k-2)-3(k-1))^2
        %=(n-k+1)^2 + (k-1)3^2 
        = n^2 -(k-1)(2n-k-8)<n^2-1,$ where we use that $n \geq \max\{6, 2k\}$ and conclude.
        
        Since the last inequality is strict, equality cannot be achieved by any matching of size at least two. In the case $k = 1$, the extremal tree must contain an edge whose endpoints each have degree $\frac{n}{2}$; the balanced double-broom is the unique example.
    \end{claimproof}
    
    Now let $T$ be a tree of order $n$ such that $T \not =\DB_{\floorfrac{n}{2}-1,\ceilfrac n2 -1}$. We know that $|\mathcal M(T)| \leq F_n$, the $n$th Fibonacci number. Thus,
    $$Z_{\phi_2}(T) \le F_n c^{\floorfrac{n^2}{4}-1}<c^{\floorfrac{n^2}{4}} < Z_{\phi_2}\left(\DB_{\floorfrac{n}{2}-1,\ceilfrac n2 -1}\right)$$
    where the second inequality requires our assumption that $c$ is sufficiently large in terms of $n$.
    \end{proof}
    
    \begin{proof}[Proof for {\hyperref[thm:exp_max34]{$\phi_3$}}]
    For $\phi_3(i,j) = (i+j)^c$, the leading term of $Z_{\phi_3}$ does not correspond with a unique graph. We first prove a general claim about maximizing the product of a set of integers, which follows from Karamata's inequality.
    
    \begin{claim}\label{clm:maxprod_whenbalanced}
        Assume integers $a_1,a_2, \ldots, a_k$ sum to $qk+r$, where $0 \le r <k$.
        Then 
        $\prod_{i=1}^k a_i \le (q+1)^r q^{k-r}$ and equality occurs only if $r$ of the $a_i$s equal $q+1$ and $k-r$ of them equal $q.$
    \end{claim}
    \begin{claimproof}
        Let the $a_i$s be ordered such that $a_1 \ge a_2 \ge \ldots \ge a_k.$
        Then $\sum_{i=1}^{\ell} a_i \ge \ell(q+1)$ for every $\ell \le r.$
        If not, $a_{i} \le q$ for every $i \ge \ell$ and thus 
        $\sum_{i=1}^{k} a_i \le \ell(q+1)-1 + (k-\ell)q<kq+\ell \le kq+r,$ a contradiction.
        Similarly, the sum of the $\ell \le k-r$ smallest values among the $a_i$s is bounded by $\ell q.$
        This implies that the sequence $(a_1,a_2 \ldots, a_k)$ majorizes the sequence
        $$\left(\underbrace{q+1,q+1, \ldots, q+1}_{r },\underbrace{q,q, \ldots, q}_{k-r } \right).$$
        By Lemma~\ref{lem:karamata} applied to the concave function $\log(x),$ we conclude that 
        $\sum_i \log(a_i) \le r\log(q+1)+(k-r) \log(q).$
        % Alternatively, if $a_1 \ge a_k +2$, then replacing $(a_1,a_k)$ with $(a_1-1,a_k+1)$ implies that the product $\prod a_i$ increased.
    \end{claimproof}

    \begin{claim}\label{clm:maxprodsum}
    Let $T$ be a tree of order $n \ge 4$ and $M$ a matching in $T$.
    Then $$\prod_{uv \in M} \left(\deg(u)+\deg(v)\right) \le 
    \begin{cases}
    3^2\cdot 4^{(n-4)/2} & \text{ if n is even,}\\
    3\cdot 4^{(n-3)/2} & \text{ if n is odd.}\\
    \end{cases}$$
    Equality is possible only if $M$ is a perfect matching (for $n$ even) or strong almost-perfect matching (for $n$ odd).
    Moreover, for all other matchings the upper bound can be reduced by a factor of $2$.
    \end{claim}
    \begin{claimproof}[Proof of Claim]
        Let $V=\{v_1, \ldots, v_n\}$ be an arbitrary permutation of the vertices and $d_i = \deg(v_{2i-1})+\deg(v_{2i})$ (observe that $v_{2i-1}v_{2i}$ is not necessarily an edge of $T$).
        If $n=2k$ is even, then  $\sum_{i=1}^k d_i = 2n-2 = 4k-2$ by the handshaking lemma and by Claim~\ref{clm:maxprod_whenbalanced}, $\prod_{i=1}^k d_i \leq 3^2 \cdot 4^{k-2}.$
        Equality happens only when all $d_i$ are $3$ or $4$. If $M$ is not a perfect matching and covers vertices $\{v_1, \dots, v_{2m}\}$ (without loss of generality), then $2m < 2k$ so we may divide both sides by $d_k \geq 2$ to get the final part of the claim.
        
        If $n=2k+1$ is odd, then $\sum_{i=1}^k d_i\le 2(n-1)-1=4k-1$ and $\prod_{i=1}^k d_i \le 3\cdot 4^{k-1}$ by Claim~\ref{clm:maxprod_whenbalanced}.
        Similar arguments hold for equality and when $M$ is not an SAPM as in the even case.
    \end{claimproof}
    We say $uv$ is an $(a,b)$-edge if $\deg(u) = a, \deg(v) = b$. To determine which tree achieves equality in the previous claim, a trivial but crucial observation is that $\deg(u)+\deg(v)=4$ if and only if $uv$ is a $(2,2)$- or $(3,1)$-edge.
    The path $P_n$ is a tree that attains equality in Claim~\ref{clm:maxprodsum}, but not the unique one.
    
    First, consider the case where $n=2k+1$ is odd.
    The leading terms for $Z_{\phi_3}(P_n)$ are $2\cdot (3\cdot 4^{k-1})^c+(k-1)\cdot (3^2 \cdot 4^{k-2})^c.$
   
    Let $T \not= P_n$ be a tree of order $n$.
    The (strong) almost-perfect matching(s) for which equality is attained in Claim~\ref{clm:maxprodsum} consist of $(3,1)$- and $(2,2)$-edges with exactly one $(2,1)$-edge.
    Observe that the extremal tree must then consist of a central path (the diameter) with pendant leaves connected to the internal vertices.
    There are at most two choices for an \sapm\ that attains equality in Claim~\ref{clm:maxprodsum}.
    If there is only one $(2,1)$-edge in $T$ (an example is presented in Figure~\ref{fig:oneconfigurationtoexplain} on the left), then there are two leaves with the same neighbor of degree $3$ and hence only $2$ \apm's in total (since an \apm\ must contain exactly one of the two neighboring leaves). 
    That implies whenever $c$ is sufficiently large that
    $$Z_{\phi_3}(T)< 2\cdot (3\cdot 4^{k-1})^c+F_n \cdot (3\cdot 4^{k-1}/2)^c<Z_{\phi_3}(P_n).$$ 
    
    If there are two $(2,1)$-edges in $T$, let $u$ and $v$ be the two leaves at the end of the diameter. Observe that we cannot have perfect matchings of both $T \setminus \{u\}$ and $T \setminus \{v\}$; this is because $T$ has a $(3,1)$-edge (since $T \neq P_n$) which must be contained in both matchings. Deleting the endpoints of the $(3,1)$-edge splits the tree into two components, both of which have odd order in one of $T \setminus \{u\}$ or $T \setminus \{v\}$. 
    So, in this case there is at most one matching for which equality in Claim~\ref{clm:maxprodsum} holds. Thus the leading term in $Z_{\phi_3}(T)$ is smaller than $2\cdot (3\cdot 4^{k-1})^c$ and the result follows again.
    
    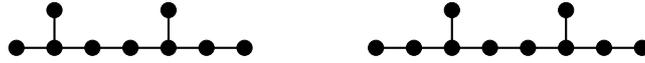
\begin{figure}[ht]

		\begin{center}
		\begin{tikzpicture}[scale=1.4]
			{
			    \draw[thick] (3.5,0) -- (0.5,0);
				\foreach \x in {1,2,...,7}{\draw[fill] (0.5*\x,0) circle (0.05);}
				\foreach \x in {2,5}{\draw[thick] (0.5*\x,0) -- (0.5*\x,0.5);
				\draw[fill] (0.5*\x,0.5) circle (0.05);
				}

			}
			\end{tikzpicture}\quad \hspace{1cm}
			\begin{tikzpicture}[scale=1.4]
			{
			    \draw[thick] (3.5,0) -- (.,0);
				\foreach \x in {0,1,2,...,7}{\draw[fill] (0.5*\x,0) circle (0.05);}
				\foreach \x in {2,5}{\draw[thick] (0.5*\x,0) -- (0.5*\x,0.5);
				\draw[fill] (0.5*\x,0.5) circle (0.05);
				}

			}
			\end{tikzpicture}
		\end{center}
\caption{Trees for which $Z_{\phi_3}(T)$ and $Z_{\phi_3}(P_n)$ have the same leading term}\label{fig:oneconfigurationtoexplain}
\end{figure}
    
    Next, we consider the case where $n$ is even.
    A tree $T$ that attains equality in Claim~\ref{clm:maxprodsum} has maximum degree at most $3$, every vertex of degree $3$ has a neighboring leaf, and $T$ has a perfect matching (see the right hand side of Figure~\ref{fig:oneconfigurationtoexplain} for an example).
    As before, $T$ consists of a diameter with pendant leaves.
    The second largest contribution to $Z_{\phi_3}(T)$ would be due to a strong almost-perfect matching $M$, for which $\prod_{uv \in M} \left(\deg(u)+\deg(v)\right) \le 2^{n-2}$ by Claim~\ref{clm:maxprodsum}, and equality occurs only if the avoided pair are the endpoints of the diameter, and the degree $3$ vertices in $M$ are paired with leaves.
    
    Let the diameter have vertices $v_1, v_2, \dots, v_k$ where $\{v_1, v_k\}$ is the avoided pair, and let $i$ be minimum such that $\deg(v_i) = 3$. In order for $M$ to cover the vertices $v_2, \dots, v_{i-1}$, we must have $i$ even. But $T$ has a perfect matching, and in order for the perfect matching to cover $v_1, v_2, \dots, v_{i-1}$, we must have $i$ odd - a contradiction unless $T$ contains no vertices of degree 3, i.e. $T = P_n$.
    \end{proof}

    \begin{proof}[Proof for {\hyperref[thm:exp_max34]{$\phi_4$}}]
    Finally, to prove the result for $\phi_4(i,j) = (ij)^c$ we use the following claim.
    \begin{claim}
        Let $T$ be a tree on $n \ge 2$ vertices.
        Then the product of the degrees of all vertices is at most $2^{n-2}$, with equality if and only if $T$ is a path $P_n$.
    \end{claim}
    \begin{claimproof}
        A tree has at least $2$ leaves, so by the handshaking lemma the other $n-2$ vertices have degree $2$ on average.
        The claim follows from the AM-GM inequality and the fact that the path is the only tree with exactly two leaves.
    \end{claimproof}
    As such, the leading term of $Z_{\phi_4}(P_n)$ is $\left( 2^{n-2}\right)^c$, while for any other tree $T$ of order $n$, we have $Z_{\phi_4}(T) \le F_n \left( 3\cdot 2^{n-4}\right)^c,$ which is smaller than $\left( 2^{n-2}\right)^c$ once e.g. $c>2n.$
\end{proof}

\begin{section}
{Further Directions}\label{sec:future}
\end{section}

We conclude with some further open problems.

\subsection{Maximum number of strong $k$-almost-perfect matchings for $k\ge 3$}\label{sec:k-sapm}

We can generalize the notion of an almost-perfect matching to a {\em $k$-almost-perfect matching} (\kapm), a matching which covers all but $k$ vertices. Similarly, a {\em strong $k$-almost-perfect matching} (\ksapm) covers all but $k$ leaves. Observe that $k \in \{1, 2\}$ recovers our original notion of {\apm} and \sapm. For $k \geq 3$, we can consider the question of which trees have the maximum number of {\ksapm}s.

Recall a spider is the unique tree of order $n$ whose legs all have length $2$, except possibly one leg of length $1$ when $n$ is even and a spider-trio is a tree constructed from three spiders whose center vertices are attached to a single new vertex.
In Theorem~\ref{thr:even_sapm_gen}, for $n$ sufficiently large, we showed that the spider-trio maximizes the number of {\sapm}s. 
For $n$ sufficiently large, one could hope that a generalized 
$(k+1)$-spider, presented in Figure~\ref{fig:GoodConstruction_k_sapm} for $k=7$, is extremal.
This might be true for small $k$, but is false for $k \ge 13$. When $k=13$ 
a tree consisting of spiders connected to a path has more {\ksapm}s than the $(k+1)$-spider. Asymptotically, however, the generalized $(k+1)$-spider has the optimal order.

\begin{figure}[ht]

\begin{center}

        \begin{tikzpicture}[x=0.6cm, y=0.6cm]
    {
    \foreach \t in {0,1,2,3,4,5,6,7}{  
	\foreach \x in {0,1,2,3,4}{\draw[thick] (1+2.5*\t,2) -- (0.5*\x+2.5*\t,1);}	
	\foreach \x in {0,1,2,3,4}{\draw[thick] (0.5*\x+2.5*\t,0) -- (0.5*\x+2.5*\t,1);}	
	
    \foreach \x in {0,1,2,3,4}{\draw[fill] (0.5*\x+2.5*\t,1) circle (0.1);}
    \foreach \x in {0,1,2,3,4}{\draw[fill] (0.5*\x+2.5*\t,0) circle (0.1);}
    \draw[fill] (1+2.5*\t,2) circle (0.1);
    }
    
    \foreach \t in {0,1,2,3,4,5,6,7}{  
	\foreach \x in {0,1,2,3,4}{\draw[thick] (1+2.5*\t,2) -- (9.75,3);}	
	}
    
	}

    \draw[fill] (9.75,3) circle (0.1);
    \end{tikzpicture}\\
\end{center}
\caption{A generalized $8$-spider}
\label{fig:GoodConstruction_k_sapm}
\end{figure}

\begin{proposition}\label{lem:spider-leading-term}
    The number of strong $k$-almost-perfect matchings in a tree of order $n$ is $\Theta_k(n^k).$
\end{proposition}
\begin{proof}
    The number of {\ksapm}s is bounded by 
    $\binom{n}{k}=O_k(n^k)$.
    For the lower bound, let $T$ be the union of $k+1$ odd spiders
    all of order $2a+1$, whose centers are connected with an additional vertex $v$
    as in Figure~\ref{fig:GoodConstruction_k_sapm}. 
    Then $n=(k+1)(2a+1)+1$ and the number of {\ksapm}s
    in $T$ is $(k+1)a^k \sim \frac{n^k}{2^k(k+1)^{k-1}}=\Omega_k(n^k).$ 
\end{proof}

We can also determine some structural properties that the extremal tree(s) must have.

\begin{proposition}\label{lem:noleavesatdist2ingen}
    If $T$ contains two leaves that share a common neighbor, then the number of strong $k$-almost-perfect matchings is at most $2\binom{n}{k-1}=O(n^{k-1}).$
\end{proposition}
\begin{proof}
    If two leaves, $\ell_1$ and $\ell_2,$ have the same neighbor, then no matching can contain both of them. Thus, in any {\kapm}
    at least one of the two is in the avoided set of $k$ leaves. 
    For both $\ell_1$ and $\ell_2$, the number of {\kapm}s is bounded 
    by the number of ways to choose the remaining $k-1$ leaves to avoid.
\end{proof}

As a consequence, for $n$ sufficiently large, we can partition the vertices
into three sets: the set of leaves $L$, the set of (unique) neighbors of the leaves 
$N := N(L)$ and the set of vertices $X$ that are at distance at least two from any leaf.
Observe that $\lvert L\rvert = \lvert N \lvert$ and $\lvert X \rvert + 2\lvert L \rvert =n$. For a subset of vertices $S$, let $T[S]$ denote the subtree induced by $S$.

\begin{proposition}\label{lem:noedgeinT[N]}
    The number of strong $k$-almost-perfect matchings that contain an edge in $T[N]$ is $O_k(n^{k-1}).$
\end{proposition}

\begin{proof}
    A {\ksapm} $M$ is uniquely determined by the avoided set of $k$ leaves.
    If there is an edge in $T[N] \cap M$, then the endpoints have neighbors in $L$ which must also be in the avoided set, together with $k-2$ other leaves.
    There are fewer than $n-1$ edges in $T[N]$ and $O(n^{k-2})$ choices for the other $k-2$ leaves, which results in a total of $O(n^{k-1})$ {\ksapm}s that contain an edge in $T[N]$.
\end{proof}

Thus, there are few {\ksapm}s that contain an edge between two vertices in $N$.
We suspect the extremal trees are spiders connected to a relatively
small $X$, but the structure of $X$ seems difficult to characterize.
Surprisingly, when $T[X]$ is a path, the optimal number of
leaves in each spider is not monotone as the following example illustrates.
 \begin{examp}
    When $T$ is the union of $8$ spiders whose centers are connected with a path, such that $T$ is symmetric, the number of strong $4$-\apm's is maximized when the ratios of their sizes (counted from outer to inner spiders) is approximately $\frac{27}{16} \colon 1 \colon \frac{9}{16} \colon \frac{3}{4}.$
\end{examp}

\begin{figure}[ht]

\begin{center}

    \begin{tikzpicture}[x=0.6cm, y=0.6cm]

    \foreach \t in {0,1,2,3,4,5,6,7}{  
	\foreach \x in {0,1,2}{\draw[thick] (1+2.5*\t,2) -- (\x+2.5*\t,1);}				
	\foreach \x in {0,1,2}{\draw[thick] (\x+2.5*\t,1) -- (\x+2.5*\t,0);}
	
    \foreach \x in {0,1,2}{\draw[fill] (\x+2.5*\t,1) circle (0.1);}
	\foreach \x in {0,1,2}{\draw[fill] (\x+2.5*\t,0) circle (0.1);}
    \draw[fill] (1+2.5*\t,2) circle (0.1);
	}
	
	\draw[thick] (1,2) -- (18.5,2);
	
	\draw [decorate,decoration={brace,amplitude=4pt},xshift=0pt,yshift=0pt] (2,-0.2)--(0,-0.2)  node [black,midway,yshift=-0.3cm]{$a$};
	
	\draw [decorate,decoration={brace,amplitude=4pt},xshift=0pt,yshift=0pt] (4.5,-0.2)--(2.5,-0.2)  node [black,midway,yshift=-0.3cm]{$b$};
	
	\draw [decorate,decoration={brace,amplitude=4pt},xshift=0pt,yshift=0pt] (7,-0.2)--(5,-0.2)  node [black,midway,yshift=-0.3cm]{$c$};
	
	\draw [decorate,decoration={brace,amplitude=4pt},xshift=0pt,yshift=0pt] (9.5,-0.2)--(7.5,-0.2)  node [black,midway,yshift=-0.3cm]{$d$};
	
	\draw [decorate,decoration={brace,amplitude=4pt},xshift=0pt,yshift=0pt] (12,-0.2)--(10,-0.2)  node [black,midway,yshift=-0.3cm]{$d$};
	
	\draw [decorate,decoration={brace,amplitude=4pt},xshift=0pt,yshift=0pt] (14.5,-0.2)--(12.5,-0.2)  node [black,midway,yshift=-0.3cm]{$c$};
	
	\draw [decorate,decoration={brace,amplitude=4pt},xshift=0pt,yshift=0pt] (17,-0.2)--(15,-0.2)  node [black,midway,yshift=-0.3cm]{$b$};
	
	\draw [decorate,decoration={brace,amplitude=4pt},xshift=0pt,yshift=0pt] (19.5,-0.2)--(17.5,-0.2)  node [black,midway,yshift=-0.3cm]{$a$};
	
	\end{tikzpicture}\\
\end{center}
\caption{A construction where $T[X]=P_8$}
\label{fig:8spiders}
\end{figure}
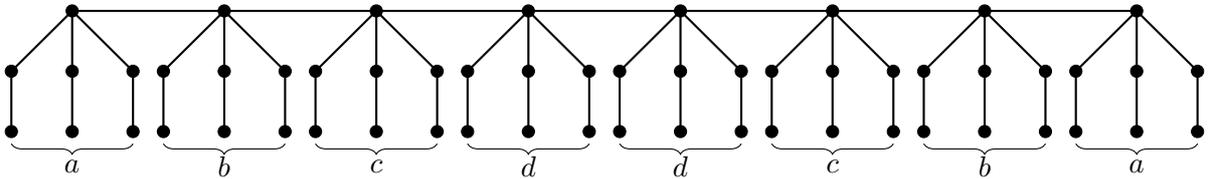

    To see this, let the number of leaves of the $8$ spiders be $a,b,c,d,d,c,b,a$ respectively, as presented in Figure~\ref{fig:8spiders}.
    Then $2(a+b+c+d)=\frac{n-8}{2}.$
    A strong $4$-\apm\ can only be obtained when taking four leaves from different spiders in such a way that the centers of the other four form a matching in the path.
    The total number of choices for these leaves is equal to $(b^2 + (2c + 2d)b + d^2)a^2 + 2c((c + 2d)b + d^2)a + c^2d^2.$
    Using a small computer verification\footnote{\url{https://github.com/StijnCambie/MRC_trees_projects}, document OptimalDistributionBehaviour.}, we conclude that this homogenous multivariate polynomial is maximized when $(a \colon b \colon c \colon d) = \left(\frac{27}{16} \colon 1 \colon \frac{9}{16} \colon \frac{3}{4}\right).$
    
The preceding discussion may provide some intuition to tackle the following question:
\begin{question}
For $k\ge 3$, characterize the trees which maximize the number of strong $k$-almost-perfect matchings.
\end{question}

We can generalize the notion of {\ksapm} even further to that of a {\em connected matching}, which is a matching $M$ in a tree $T$ such that $T[V(M)]$ is a connected graph. 
The notion of connected matching (among some other variants) was defined in~\cite{GHHL05}.
Observe that every {\ksapm} is a connected matching. Conversely, when $n$ is odd, a connected matching with $\frac{n-1}2$ edges is an \sapm.
One could ask similar questions about characterizing the extremal trees for connected matchings. 
It may be interesting to consider these type of questions for connected matchings in general graphs as well, or classes such as $d$-regular graphs which are well-studied in matching theory.

\subsection{Weighted Hosoya index}

Regarding applications to the weighted Hosoya index, it would be interesting to determine the extremal trees for various other choices of degree-based weight function. In particular, can we characterize when the path and the star are the maximizer and minimizer, respectively? 

For example, consider the function 
$$\phi(i,j) = \varphi^{16\lfloor \frac{ij}{16}\rfloor}$$
where $\varphi$ is the golden ratio. While this choice of weight function may seem rather artificial, it is interesting to note that neither $P_n$ nor $S_n$ are extremal; one can check that $Z_\phi(P_n) = F_n \sim \varphi^n$ and $Z_\phi(S_n) \sim (n-1)\varphi^{n-1}$. Here $F_n$ is the $n^{th}$ Fibonacci number.
In fact, a weight of at least approximately $\varphi^{8n/3}$ can be achieved by the construction in Figure~\ref{fig:largeZphi} below.

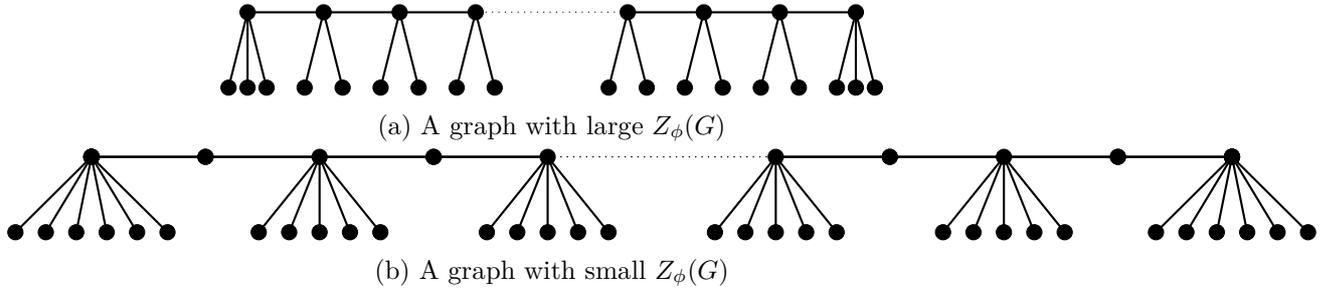
\begin{figure}[ht]
\centering
\begin{minipage}[b]{.9\linewidth}
\begin{center}
    \begin{tikzpicture}[scale=0.7]
    {
	\foreach \x in {0,1,2,3,5,6,7,8}{
	\foreach \a in {-0.25,0.25}{
	\draw[thick] (\x+\a,0) -- (\x,1);
	\draw[fill] (\x+\a,0) circle (0.1);}	\draw[fill] (\x,1) circle (0.1);		
	}
	
    \foreach \x in {0,8}{
    \draw[fill] (\x,0) circle (0.1);
    \draw[thick] (\x,0) -- (\x,1);}
	\draw[thick] (0,1) -- (3,1);
	\draw[dotted] (5,1) -- (3,1);
	\draw[thick] (8,1) -- (5,1);		
	}
	\end{tikzpicture}\\
\subcaption{A graph with large $Z_{\phi}(G)$}
\label{fig:largeZphi}
\end{center}
\end{minipage}\quad\begin{minipage}[b]{.9\linewidth}

\begin{center}
    \begin{tikzpicture}[scale=0.7]
    {
	\foreach \x in {2.4,4.8,7.2,9.6}{
	\foreach \a in {-0.8,-0.4,0,0.4,0.8}{
	\draw[thick] (\x+\a,0) -- (\x,1);
	\draw[fill] (\x+\a,0) circle (0.1);}	\draw[fill] (\x,1) circle (0.1);		
	}
    \foreach \x in {0,12}{
    \foreach \a in {-1,-0.6,...,0.6,1}{
    \draw[fill] (\x,1) circle (0.1);
	\draw[thick] (\x+\a,0) -- (\x,1);
	\draw[fill] (\x+\a,0) circle (0.1);
	}
	\foreach \x in {1,2,5,4}{
	\draw[fill] (2.4*\x-1.2,1) circle (0.1);
	}
	\draw[thick] (12,1) -- (7.2,1);
	\draw[dotted] (4.8,1) -- (7.2,1);
	\draw[thick] (4.8,1) -- (0,1);
	}}
	\end{tikzpicture}\\
\subcaption{A graph with small $Z_{\phi}(G)$}
\label{fig:smallZphi}
\end{center}
\end{minipage}
\caption{Two trees indicating that $P_n$ and $S_n$ are not extremal}\label{fig:2treesforZphi}
\end{figure}

On the other hand, a weight of at most $8^{n/7}< \varphi^{3n/4}$ is given by the construction in Figure~\ref{fig:smallZphi}.

One might try to generalize the results of Section~\ref{sec:HosoyaIndex} to answer the following:
\begin{question}
Can we fully characterize the class of functions $\phi$ for which $Z_\phi(T)$ is minimized by the star, or maximized by the path?
\end{question}

\paragraph{Acknowledgments} The authors would like to express their gratitude to the American Mathematical Society for organizing the Mathematics Research Community workshops where this
work began, in the workshop ``Trees in Many Contexts,'' and to Wanda Payne for contributing to the initial discussions. This event was supported by the National Science Foundation under Grant Number DMS $1916439$. 

The first author has been supported by the Institute for Basic Science (IBS-R029-C4) and a postdoctoral fellowship by the Research Foundation Flanders (FWO) with grant number 1225224N. The fourth author was supported by the Knut and Alice Wallenberg Foundation (KAW 2017.0112) and the Swedish research council (VR), grant  2022-04030.

%Fill the list of references in the alphabetical order here.
% Adopt the style like:

%\bibitem{bib1}
%{\small {\sc J. A. Baker:} {\it Isometries in normed spaces.}
%Amer. Math. Monthly, {\bf 78} (1971), 655--658.}

%\bibitem{bib2}
% {\small {\sc A. Marshall, I. Olkin: }{\it Inequalities$:$ Theory of
%Majorization and Its Applications.} Academic Press, New York, 1979.}

%\vspace{1cc}
\newpage

%Fill author(s) affiliation(s), address(es) and emails here:
%example is shown
\noindent
{\bf Stijn Cambie}\\
Department of Computer Science,\\
KU Leuven Campus Kulak-Kortrijk,\\
Kortrijk, Belgium\\
E-mail: {\it stijn.cambie@hotmail.com}

\vspace{0.1cc}

\noindent
{\bf Bradley McCoy}\\
Department of Computer Science,\\
James Madison University,\\
Harrisonburg, VA, USA\\
E-mail: {\it mccoy2ba@jmu.edu}

\vspace{0.1cc}

\noindent
{\bf Gunjan Sharma}\\
Applied Mathematics,\\
Illinois Institute of Technology,\\
Chicago, IL, USA\\
E-mail: {\it gsharma7@hawk.iit.edu}

\vspace{0.1cc}

\noindent
{\bf Stephan Wagner}\\
Institute of Discrete Mathematics,\\
TU Graz,\\
Graz, Austria\\
and\\
Department of Mathematics,\\
Uppsala University,\\
Uppsala, Sweden\\
E-mail: {\it stephan.wagner@tugraz.at} 

\vspace{0.1cc}

\noindent
{\bf Corrine Yap}\\
School of Mathematics,\\
Georgia Institute of Technology,\\
Atlanta, GA, USA\\
E-mail: {\it math@corrineyap.com}

\end{document}